\documentclass[11pt,reqno]{amsart}
\usepackage{url, multirow}

\newcommand{\pFq}[5]{\ensuremath{{}_{#1}F_{#2} \left( \genfrac{}{}{0pt}{}{#3}
{#4} \bigg| {#5} \right)}}
\newcommand{\df}{\dfrac}
\newcommand{\tf}{\tfrac}

 \renewcommand{\a}{\alpha}
\renewcommand{\b}{\beta}

\renewcommand{\d}{{\delta}}
\newcommand{\g}{\gamma}
\newcommand{\G}{\Gamma}
\renewcommand{\l}{\lambda}

\renewcommand{\(}{\left\(}
\renewcommand{\)}{\right\)}
\renewcommand{\[}{\left\[}
\renewcommand{\]}{\right\]}
\let\dotlessi=\i
\let\dotlessi=\i
\renewcommand{\pmod}[1]{\,(\textup{mod}\,#1)}
\numberwithin{equation}{section}
 \theoremstyle{plain}
\newtheorem{theorem}{Theorem}[section]

\newtheorem{corollary}[theorem]{Corollary}

\newtheorem{remark}[theorem]{Remark}

   \makeatletter
\def\proof{\@ifnextchar[{\@oproof}{\@nproof}}
\def\@oproof[#1][#2]{\trivlist\item[\hskip\labelsep\textit{#2 Proof of\
#1.}~]\ignorespaces}
\def\@nproof{\trivlist\item[\hskip\labelsep\textit{Proof.}~]\ignorespaces}
\makeatother

\begin{document}
\title{Sums of squares and products of Bessel functions}
\author{Bruce C.~Berndt, Atul Dixit, Sun Kim, and Alexandru Zaharescu}
\address{Department of Mathematics, University of Illinois, 1409 West Green
Street, Urbana, IL 61801, USA} \email{berndt@illinois.edu}
\address{Department of Mathematics, Indian Institute of Technology Gandhinagar, Palaj, Gandhinagar 382355, Gujarat, India}\email{adixit@iitgn.ac.in}
\address{Department of Mathematics, University of Illinois, 1409 West Green
Street, Urbana, IL 61801, USA} \email{sunkim2@illinois.edu}
\address{Department of Mathematics, University of Illinois, 1409 West Green
Street, Urbana, IL 61801, USA and \newline
		Simion Stoilow Institute of Mathematics of the Romanian
		Academy, P.O. Box 1--764, RO--014700 Bucharest, Romania.} \email{zaharesc@illinois.edu}
%\begin{center}
%\today
%\end{center}
\begin{abstract}
Let $r_k(n)$ denote the number of representations of the positive integer $n$ as the sum of $k$ squares. We rigorously prove for the first time a Vorono\"{\dotlessi} summation formula for $r_k(n), k\geq2,$ proved incorrectly by A. I. Popov and later rediscovered by A. P. Guinand, but without proof and without conditions on the functions associated in the transformation. Using this summation formula we establish a new transformation between a series consisting of $r_k(n)$ and a product of two Bessel functions, and a series involving $r_k(n)$ and the Gaussian hypergeometric function. This transformation can be considered as a massive generalization of well-known results of G. H. Hardy, and of A. L. Dixon and W. L. Ferrar, as well as of a classical result of A. I. Popov that was completely forgotten. An analytic continuation of this transformation yields further useful results that generalize those obtained earlier by Dixon and Ferrar.
\end{abstract}


\maketitle

\section{Introduction}
Infinite series involving arithmetic functions and Bessel functions are instrumental in studying some notoriously difficult problems in analytic number theory, for example, the circle and the divisor problems. As mentioned by G. H. Hardy \cite[p.~266]{hardyqjpam1915}, S.~Wigert \cite{wigert} was the first mathematician to recognize the importance of series of Bessel functions in analytic number theory. Since then, several mathematicians have studied, and continue to study, such series, for example, with the point of view of understanding and improving the order of magnitude of error terms associated with the summatory functions of certain arithmetic functions. A prime tool in making the connection between a summatory function and certain series of Bessel functions is the Vorono\"{\dotlessi} summation formula associated with the corresponding arithmetic function.

Let $r_k(n)$ denote the number of representations of a positive integer $n$ as the sum of $k$ squares, where different signs and different orders of the summands give distinct representations. The ordinary Bessel function $J_{\nu}(z)$ of order $\nu$ is defined by \cite[p.~40]{watson-1966a}
\begin{align}\label{sumbesselj}
	J_{\nu}(z):=\sum_{m=0}^{\infty}\frac{(-1)^m(z/2)^{2m+\nu}}{m!\Gamma(m+1+\nu)}, \quad |z|<\infty.
	\end{align}
We record the Vorono\"{\dotlessi} summation formula associated with $r_2(n)$, sometimes known as the Hardy--Landau summation formula, in the form given in \cite[p.~274]{landau} (or \cite[Thm.~A]{dixfer2}).

\begin{theorem}
If $0\leq\a<\b$ and $h(y)$ is real and of bounded variation in $(\a,\b)$, then
\begin{equation}\label{hlsf}
\sum_{\a\leq n\leq\b}r_2(n)\tfrac{1}{2}(h(n-0)+h(n+0))=\pi\sum_{n=0}^{\infty}r_2(n)\int_{\a}^{\b}h(y)J_{0}(2\pi\sqrt{ny})\, dy,
\end{equation}
where, if $n=\b$, the coefficient of $r_2(\b)$ is taken to be  $\tfrac{1}{2}h(\b-0)$; if $n=\a\neq 0$, the coefficient of $r_2(\a)$ is taken to be $\tfrac{1}{2}h(\a+0)$; and if $n=\a=0$, the coefficient of $r_2(0):=1$ is taken to be $h(0+)$.
\end{theorem}

A.~L.~Dixon and W.~L.~Ferrar \cite[eqs.~(2.2), (2.3)]{dixfer2} extended this theorem to include the case when $\b=\infty$ and obtained the following result.

\begin{theorem}\label{hlinf}
If $h(y), h'(y)$ and $h''(y)$ are bounded in $(0,\infty)$, and are $O(\textup{exp}(-y^u))$ for $y$ large and $u>0$, then
\begin{equation}\label{r2vsfinf}
\sum_{n=0}^{\infty}r_2(n)h(n)=\pi\sum_{n=0}^{\infty}r_2(n)\int_{0}^{\infty}h(t)J_{0}(2\pi\sqrt{nt})\, dt.
\end{equation}
\end{theorem}

To state Dixon and Ferrar's application of Theorem \ref{hlinf}, we need to define two further Bessel functions. The modified Bessel function of the first kind of order $\nu$ is defined by \cite[p.~77]{watson-1966a}
	\begin{equation}\label{besseli}
I_{\nu}(z)=
\begin{cases}
e^{-\frac{1}{2}\pi\nu i}J_{\nu}(e^{\frac{1}{2}\pi i}z), & \text{if $-\pi<$ arg $z\leq\frac{\pi}{2}$,}\\
e^{\frac{3}{2}\pi\nu i}J_{\nu}(e^{-\frac{3}{2}\pi i}z), & \text{if $\frac{\pi}{2}<$ arg $z\leq \pi$,}
\end{cases}
\end{equation}
where $J_{\nu}(z)$ is the ordinary Bessel function of order $\nu$ defined in \eqref{sumbesselj}.
	The modified Bessel function of the second kind is defined by \cite[p.~78, eq.~(6)]{watson-1966a},
\begin{equation}\label{besselk}
K_{\nu}(z):=\frac{\pi}{2}\frac{I_{-\nu}(z)-I_{\nu}(z)}{\sin\nu\pi}.
\end{equation}

Using Theorem \ref{hlinf}, Dixon and Ferrar \cite[eq.~(3.12)]{dixfer2} showed that, for Re$(\sqrt{\b})>0$ and Re$(\nu)>0$,
\begin{equation}\label{r2dix}
\sum_{n=0}^{\infty}r_2(n)n^{\frac{\nu}{2}}K_{\nu}(2\pi\sqrt{n\b})=\frac{\b^{\frac{\nu}{2}}\G(\nu+1)}{2\pi^{\nu+1}}
\sum_{n=0}^{\infty}\frac{r_{2}(n)}{(n+\b)^{\nu+1}}.
\end{equation}
If we set $\nu=\tf12$ in \eqref{r2dix} and employ the formula \cite[p.~80, eq.~(13)]{watson-1966a}
\begin{equation}\label{kspl}
K_{\frac{1}{2}}(z)=\sqrt{\dfrac{\pi}{2z}}e^{-z},
\end{equation}
after a change of variable, we deduce a result of
Hardy \cite[eq.~(2.12)]{hardyqjpam1915}
%\begin{equation}
%{\sum_{0<n\leq x}}^{'}r_2(n)=\pi x-1+\sum_{k=1}^{\infty}r_2(k)\left(\frac{x}{k}\right)^{1/2}J_{1}(2\pi\sqrt{kx}).
%\end{equation}
\begin{equation*}
\sum_{n=1}^{\infty}r_2(n)e^{-s\sqrt{n}}=\df{2\pi}{s^2}-1+2\pi s\sum_{n=1}^{\infty}\df{r_2(n)}{(s^2+4\pi^2n)^{3/2}},
\end{equation*}
where $\text{Re } s>0$.  This was the primary identity that Hardy used to derive a lower bound for the error term in the famous \emph{circle problem}.

In \cite{dixfer2}, Dixon and Ferrar also obtained a generalization of \eqref{r2dix}, namely for Re$(\sqrt{\b})>0$, $\d\geq 0$, and $\nu$ arbitrary,
\begin{equation}\label{dixonferrar}
\b^{\nu/2}\sum_{n=0}^{\infty}\frac{r_2(n)}{(n+\d)^{\nu/2}}K_{\nu}(2\pi\sqrt{\b(n+\d)})
=\d^{(1-\nu)/2}
\sum_{n=0}^{\infty}\frac{r_2(n)}{(n+\b)^{(1-\nu)/2}}K_{1-\nu}(2\pi\sqrt{\d(n+\b)}).
\end{equation}
It is easy to see that if we replace $\nu$ by $-\nu$ in \eqref{dixonferrar}, let Re$(\nu)>0$, and then let $\d\to 0$, we obtain \eqref{r2dix}, with the help of \eqref{besselk} and \eqref{besseli}. If we set $\delta=a$, $\beta=b$, and  $\nu=\tf12$ in \eqref{dixonferrar}, and use \eqref{kspl}, we obtain a beautiful formula of Ramanujan \cite[p.~283, eq.~(4.21)]{hardyqjpam1915}, namely, for Re$(a)$, Re$(b)>0$,
\begin{equation}\label{ramdixfer}
\sum_{n=0}^{\infty}\frac{r_2(n)}{\sqrt{n+a}}e^{-2\pi\sqrt{b(n+a)}}=\sum_{n=0}^{\infty}\frac{r_2(n)}{\sqrt{n+b}}e^{-2\pi\sqrt{a(n+b)}}.
\end{equation}
%Here, and throughout the sequel, the prime $'$ indicates that if $x$ is an integer, then only $\frac{1}{2}r_2(x)$ contributes to the sum on the left side.
 N.~S.~Koshliakov \cite{koshsum} obtained a generalization of \eqref{dixonferrar} with $r_2(n)$ replaced by $F(n)$, the number of representations of $n$ by a binary quadratic form of discriminant $\Delta<0$. A generalization, in turn, involving the coefficients of a general Dirichlet series satisfying a functional equation, was obtained by the first author \cite[p.~343, Thm.~9.1]{III}. See also \cite[eq.~(5.5)]{bls}.

Around the same time as \cite{dixfer2} appeared, A.~I. Popov \cite[eq.~(10)]{popov1935} gave a beautiful transformation for a series involving $r_2(n)$ and a product of modified Bessel functions of the first and second kind, that is, $I_{\nu}(z)$ and $K_{\nu}(z)$, respectively. He claimed that if Re$(\nu)>0$ and $\a\geq\b>0$, then
\begin{align}\label{popovid}
&\frac{2\pi}{(\a-\b)^{\nu}}\sum_{n=1}^{\infty}r_2(n)I_{\nu}(\pi(\sqrt{n\a}-\sqrt{n\b}))K_{\nu}(\pi(\sqrt{n\a}+\sqrt{n\b}))\nonumber\\
&=\frac{\nu-\pi\sqrt{\a\b}}{\nu\sqrt{\a\b}\left(\sqrt{\a}+\sqrt{\b}\right)^{2\nu}}
+\sum_{n=1}^{\infty}\frac{r_2(n)}{\sqrt{n+\a}\sqrt{n+\b}\left(\sqrt{n+\a}+\sqrt{n+\b}\right)^{2\nu}}.
\end{align}
Later, we will prove \eqref{popovid} and more generally show that it holds for a much larger region, namely, Re$(\sqrt{\a})\geq\textup{Re}(\sqrt{\b})>0$. Assuming \eqref{popovid} for the moment, however, let
$\a\to\b^{+}$. The interchange of the order of limit and summation can be justified using the exponential decay of the summand of the series on the left side, which is discussed in detail in Section \ref{proofcoro}. Noting that, by \eqref{besseli},
\begin{equation*}
\lim_{\a\to\b^{+}}\frac{I_{\nu}(\pi(\sqrt{n\a}-\sqrt{n\b}))}{(\a-\b)^{\nu}}
=\frac{1}{\G(\nu+1)}\left(\frac{\pi\sqrt{n}}{4\sqrt{\b}}\right)^{\nu}
\end{equation*}
and, by \eqref{besselk} and \eqref{besseli},
\begin{equation*}
\lim_{x\to0}x^{\nu/2}K_{\nu}(2\pi\sqrt{x\b})=\df{\Gamma(\nu)}{2\pi^{\nu}\b^{\nu/2}},
\end{equation*}
we are led to \eqref{r2dix}. Hence, \eqref{popovid} is another generalization of \eqref{r2dix}, different from \eqref{dixonferrar}.

Popov neither gave a proof of \eqref{popovid} in \cite{popov1935}, nor did he indicate how to prove it. In what follows, we give plausible evidence for how Popov might have arrived at \eqref{popovid}.

In \cite[eq.~(3)]{popov}, he gave the following result for any positive integer $k\geq 2$:
\begin{align}\label{popres}
&\lim_{x\to 0}\frac{h(x)}{x^{k/4-1/2}}+\sum_{n=1}^{\infty}\frac{r_k(n)}{n^{k/4-1/2}}h(n)\nonumber\\
&=\frac{\pi^{k/2}}{\G(k/2)}\int_{0}^{\infty}x^{k/4-1/2}h(x)\, dx+\pi\sum_{n=1}^{\infty}\frac{r_k(n)}{n^{k/4-1/2}}\int_{0}^{\infty}h(x)J_{k/2-1}(2\pi\sqrt{nx})\, dx.
\end{align}
Note that since $J_{0}(0)=1$, \eqref{popres} reduces to \eqref{r2vsfinf} when $k=2$. Also, in \cite{popovfock}, he gave a short proof of the following beautiful integral evaluation found by V.~A.~Fock \cite[eqs.~(31), (33)]{fock}, namely, for Re$(z)>$ Re$(w)>0$ and Re$(\nu)>-3/4$,
\begin{align}\label{fock}
\int_{0}^{\infty}J_{0}(\rho x)\left(\frac{\sqrt{x^2+z^2}-\sqrt{x^2+w^2}}{\sqrt{x^2+z^2}+\sqrt{x^2+w^2}}\right)^{\nu}\frac{x\, dx}{\sqrt{x^2+z^2}\sqrt{x^2+w^2}}=I_{\nu}\left(\frac{\rho(z-w)}{2}\right)K_{\nu}\left(\frac{\rho(z+w)}{2}\right).
\end{align}
Now observe that if we let
\begin{equation}\label{eff2}
h(t):=\frac{1}{2\pi\sqrt{t+\a}\sqrt{t+\b}}\left(\frac{\sqrt{t+\a}-\sqrt{t+\b}}{\sqrt{t+\a}+\sqrt{t+\b}}\right)^{\nu}
\end{equation}
in the special case $k=2$ of \eqref{popres}, or equivalently in \eqref{r2vsfinf}, and employ \eqref{fock}, we are led to \eqref{popovid} after simplification.
%we are certain that he proved it using Theorem \ref{hlinf}, for the key step in the proof that we devised is the
%of which Popov \cite{popovfock} gave a short proof a year before \cite{popov1935} appeared.

 Several comments are in order. Firstly, we note that Popov \cite{popov} does not give any conditions on $h$ for \eqref{popres} to hold. Secondly, as we show below, his proof of \eqref{popres} is deficient. Thirdly, note that the $h$ in \eqref{eff2} does not satisfy the big-$O$ hypothesis in Theorem \ref{hlinf}. Thus the proof of \eqref{popovid} mentioned above is purely formal. Popov wrote 13 papers in mathematics, and in none of these papers is a rigorous proof of \eqref{popovid} given. %If Popov did have a rigorous proof of his result, we have, unfortunately, no way of knowing it since he has not given it in his $13$ published papers in Mathematics.

In this paper, we not only rigorously prove Popov's identity \eqref{popovid}, but we also give its massive generalization for any positive integer $k\geq 2$, with several well-known theorems in the literature arising as corollaries. We use \eqref{popres} in order to generalize \eqref{popovid}. But one needs to first obtain conditions on $h$ so that \eqref{popres} holds. The task of generalizing such results for $k=2$ to $k\geq 2$ is not straightforward, as we now explain.
%A natural question which now arises is: Can one extend \eqref{r2dix}, and more generally \eqref{popovid}, to $r_k(n)$, when $k\geq 2$? An answer to the first part of the question had already been solved by Popov himself \cite[eq.~(6)]{popov1935} in the special case $\b>0$; see Corollary \ref{cordfpop} below.

In his study of the average order of $r_2(n)$, Hardy \cite[eq.~(1.25)]{hardyqjpam1915} offered the identity
\begin{equation}\label{hqjpam1.25}
{\sum_{0<n\leq x}}^{\prime}r_2(n)=\pi x-1+\sqrt{x}\sum_{n=1}^{\infty}\frac{r_2(n)}{\sqrt{n}}J_1(2\pi\sqrt{nx}),
\end{equation}
where, here, and throughout the sequel, the prime $'$ on the summation sign indicates that if $x$ is an integer, then only one-half of the summand is counted.  Hardy \cite[p.~265]{hardyqjpam1915} acknowledged that ``the form'' of \eqref{hqjpam1.25} was suggested by Ramanujan.

Popov \cite[eq.~(2)]{popov} incorrectly generalized \eqref{hqjpam1.25} by claiming that
\begin{equation}\label{1.25an}
{\sum_{0< n\leq x}}^{\prime}r_k(n)=\frac{(\pi x)^{k/2}}{\G(1+\frac{k}{2})}-1+x^{k/4}\sum_{n=1}^{\infty}\frac{r_k(n)}{n^{k/4}}J_{k/2}(2\pi\sqrt{nx}),
\end{equation}
for any positive integer $k\geq 2$. As can be seen from the work of Oppenheim \cite[Theorem 2]{oppenheim}, when $k>2$, the series on the right-hand side in \eqref{1.25an} does not converge but is Riesz summable $(R, n, \tfrac{1}{2}k-\frac{3}{2}+\epsilon)$ for any $\epsilon>0$.

Also, it is known \cite[p.~19]{cn} that if $x>0$ and $q>\tf12(k-1)$, then
{\allowdisplaybreaks\begin{align}\label{bessel1}
\df{1}{\Gamma(q+1)}&{\sum_{0\leq n\leq x}}^{\prime}r_k(n)(x-n)^{q}\nonumber\\
&\quad\quad=\df{\pi^{k/2}x^{k/2+q}}{\Gamma(q+1+k/2)}+\left(\df{1}{\pi}\right)^{q}
\sum_{n=1}^{\infty}r_k(n)\left(\df{x}{n}\right)^{k/4+q/2}J_{k/2+q}(2\pi\sqrt{nx}),
\end{align}}%
where the series on the right-hand side converges absolutely. Here again the prime $'$ on the summation sign in \eqref{bessel1} indicates that if $x$ is an integer, then only $\frac{1}{2}r_k(x)$ is counted  in the case that $q=0$. As proved in \cite[p.~19, eq.~(67)]{cn}, the validity of \eqref{bessel1} can be extended to $q>\tfrac{1}{2}(k-3)$, which appears to be best possible. However, Popov \cite[eq.~(1)]{popov} used \eqref{bessel1} with $q>-1$ in his proof. Observe that, in particular, if $q=0$,  then \eqref{bessel1} is valid only for $k=2$, in agreement with our discussion in the previous paragraph.

Since almost all proofs of  Vorono\"{\dotlessi}-type summation formulas (such as \eqref{popres}) require that the Riesz sum identity (such as \eqref{bessel1}) holds (see, for example, \cite[p.~142]{V}), the problem of extending \eqref{hlsf} to $k\geq 2$ appears to be delicate. Moreover, in the case of \eqref{popres}, that is, when all of the associated sums and integrals are infinite, one has to take extra care.

We \cite{bdkz1} recently circumvented this issue pertaining to an analogue of \eqref{hlsf} for $k\geq 2$  by noting that the ingenious proof by
N.~S.~Koshliakov \cite{kosh1}, \cite{kosh2} of the Vorono\"{\dotlessi} summation formula for any arithmetical function whose Dirichlet series satisfies a functional equation with one simple gamma factor does not make use of \eqref{bessel1}.  Recall \cite[p.~18]{cn} that if
\begin{equation}\label{zeta}
\zeta_k(s):=\sum_{n=1}^{\infty}\df{r_k(n)}{n^s},\quad\text{Re }s>\frac{k}{2},
\end{equation}
then $\zeta_k(s)$ has an analytic continuation to the entire complex plane and satisfies the functional equation
$$\pi^{-s}\Gamma(s)\zeta_k(s)=\pi^{s-\frac{k}{2}}\Gamma(\tfrac{k}{2}-s)\zeta_k(\tfrac{k}{2}-s).$$
In particular, Koshliakov's method then rigorously gives a proof of the analogue of \eqref{hlsf} for $k\geq 2$. For more details, the reader is referred to \cite{bdkz1}. Note that in Koshliakov's work, $h$ is analytic. We rephrase in the following theorem \cite[Thms.~2.3, 2.4]{bdkz1} our extension of Koshliakov's result \cite[p.~10]{kosh1}, \cite[p.~62]{kosh2} in the case $\a\to 0$.

\begin{theorem}\label{vv}
Let
\begin{equation*}
\varphi(s):=\sum_{n=1}^{\infty}a(n)\lambda_n^{-s} \quad\text{and}\quad \psi(s):=\sum_{n=1}^{\infty}b(n)\mu_n^{-s},
\end{equation*}
where $0<\lambda_1<\lambda_2<\cdots <\lambda_n\to\infty$ and $0<\mu_1<\mu_2<\cdots<\mu_n\to\infty$, and where the abscissae of absolute convergence are $\sigma_a$ and $\sigma_a^*$, respectively. Let $\varphi(s)$ and $\psi(s)$ satisfy a functional equation of the type
\begin{equation}\label{funcequa}
\Gamma(s)\varphi(s)=\Gamma(\kappa-s)\psi(\kappa-s),
\end{equation}
for some $\kappa>0$.  Suppose that there exists a meromorphic function $\chi$ with the following properties:
\begin{align*}
\textup{(i)}& \quad\chi(s)= \Gamma(s)\varphi(s), \quad \sigma>\sigma_a,\qquad \chi(s)= \Gamma(\kappa-s)\psi(\kappa-s),\quad \sigma<\kappa-\sigma_a^*;\\
\textup{(ii)}& \quad\lim_{|\textup{Im }s|\to\infty}\chi(s)=0, \text{uniformly in every interval} -\infty<\sigma_1\leq\sigma\leq\sigma_2<\infty;\\
\textup{(iii)}& \quad\text{the poles of $\chi$ are confined to a compact set}.
\end{align*}
Let $x>0$, and let $N$ be an integer such that $\lambda_{N}<x<\lambda_{N+1}$.  Suppose that all of the poles of $\varphi(s)$ lie in the half-plane \textup{Re}$(s)>0$, and let $h(z)$ be an analytic function containing the interval $[0,x]$ in its domain of analyticity. If  the infinite series and the integrals on the right side below converge uniformly on $[0, x]$, then
\begin{equation*}
{\sum_{\lambda_n\leq x}}^{\prime}a(n)h(\lambda_n)=\lim_{a\to0}\varphi(0)h(a)+\int_0^x{Q}^{\prime}_0(t)h(t)\, dt +\sum_{n=1}^{\infty}\df{b(n)}{\mu_n^{\kappa-1}}
\int_0^x\mathcal{I}_{-1}(\mu_nt)h(t)\, dt,
\end{equation*}
where
\begin{equation*}
\mathcal{I}_q(x):=x^{(\kappa+q)/2}J_{\kappa+q}(2\sqrt{x})
\end{equation*}
and
\begin{equation*}
Q_q(x):=\df{1}{2\pi i}\int_{C_q}\df{\Gamma(s)\varphi(s)}{\Gamma(s+q+1)}x^{s+q}ds,
\end{equation*}
with $C_q$ being a positively oriented closed curve (or curves) containing all of the integrand's poles on the interior of $C_q$.
\end{theorem}

It was only recently that we learned that Guinand \cite[p.~117, Thm. 5]{guinandsumself2} derived the following summation formula.

\begin{theorem}\label{guirknsumthm}
Let $k$ be a positive integer greater than $3$ and let $m=\left\lfloor\frac{1}{2}k\right\rfloor-1$. If $F(x), F'(x)$, \newline $F''(x), \dots, F^{(2m-1)}(x)$ are integrals, and $F(x), xF'(x), x^2F''(x), \dots, x^{2m}F^{(2m)}(x)$ belong to $L^{2}(0,\infty)$, then
\begin{align}\label{guirknsum}
&\lim_{N\to\infty}\left\{\sum_{n=1}^{N}\left(1-\frac{n}{N}\right)^{m+1}r_k(n)n^{\frac{1}{2}-\frac{k}{4}}F(n)
-\frac{\pi^{\frac{k}{2}}}{\G(\frac{k}{2})}\int_{0}^{N}\left(1-\frac{x}{N}\right)^{m+1}x^{\frac{k}{4}-\frac{1}{2}}F(x)\, dx\right\}\nonumber\\
&=\lim_{N\to\infty}\left\{\sum_{n=1}^{N}\left(1-\frac{n}{N}\right)^{m+1}r_k(n)n^{\frac{1}{2}-\frac{k}{4}}G(n)
-\frac{\pi^{\frac{k}{2}}}{\G(\frac{k}{2})}\int_{0}^{N}\left(1-\frac{x}{N}\right)^{m+1}x^{\frac{k}{4}-\frac{1}{2}}G(x)\, dx\right\},
\end{align}
where
\begin{equation}\label{fgrel}
\int_{0}^{x}G(y)y^{\frac{k}{4}-\frac{1}{2}}\, dy=x^{\frac{k}{4}}\int_{0}^{\infty}y^{-\frac{1}{2}}J_{\frac{k}{2}}(2\pi\sqrt{xy})F(y)\, dy,
\end{equation}
and $G(x)$ is chosen so that it is the integral of its derivative.
\end{theorem}

 Guinand also remarks in a footnote of \cite[p.~117]{guinandsumself2} that one could prove this result with a slightly smaller value of $m$ than the one considered here.

In \cite[p.~264, eq.~(10.7)]{guinandconcord}, he also claims without proof that
\begin{align}\label{guirknsuminf}
&\sum_{n=1}^{\infty}r_k(n)n^{\frac{1}{2}-\frac{k}{4}}F(n)
-\frac{\pi^{\frac{k}{2}}}{\G(\frac{k}{2})}\int_{0}^{\infty}x^{\frac{k}{4}-\frac{1}{2}}F(x)\, dx\nonumber\\
&=\sum_{n=1}^{\infty}r_k(n)n^{\frac{1}{2}-\frac{k}{4}}G(n)
-\frac{\pi^{\frac{k}{2}}}{\G(\frac{k}{2})}\int_{0}^{\infty}x^{\frac{k}{4}-\frac{1}{2}}G(x)\, dx,
\end{align}
where
\begin{equation}\label{fgrelinf}
G(x)=\pi\int_{0}^{\infty}F(t)J_{\frac{k}{2}-1}(2\pi\sqrt{xt})\, dt.
\end{equation}
 If we let $F=h$, it can be easily seen that this formula is the same as \eqref{popres},
%\begin{equation}\label{F}
%F(x)=x^{\frac{k}{4}-\frac{1}{2}}f(\pi x),
%\end{equation}
%in \eqref{guirknsuminf},
except that Guinand's formula involves the term
$$-\frac{\pi^{\frac{k}{2}}}{\G(\frac{k}{2})}\int_{0}^{\infty}x^{\frac{k}{4}-\frac{1}{2}}G(x)\, dx,$$
 whereas \eqref{popres} contains $-\lim_{x\to 0}F(x)/x^{k/4-1/2}$. Under appropriate hypotheses on $F(x)$ and $G(x)$, we now show that
 \begin{equation}\label{b1}
 \lim_{x\to0}\df{F(x)}{x^{\frac{k}{4}-\frac12}}=\df{\pi^{\frac{k}{2}}}{\Gamma(\frac{k}{2})}\int_0^{\infty}x^{\frac{k}{4}-\frac12}G(x)dx,
 \end{equation}
 where $G(x)$ is defined by \eqref{fgrelinf}. If $F(x), G(x) \in L(0,\infty)$ and are of bounded variation at each point in $(0,\infty)$, then the Hankel transform of $G(x)$ holds \cite[p.~240]{titchfourier}, i.e.,
 \begin{equation}\label{b2}
 F(x)=\pi\int_{0}^{\infty}G(t)J_{\frac{k}{2}-1}(2\pi\sqrt{xt})\, dt.
\end{equation}
To justify taking the limit inside the integral below, we invoke the following theorem from Titchmarsh's text \cite[p.~25]{titchcomplex}. If $f(x,y)$ is continuous on the rectangle $a\leq x\leq b, \a\leq y\leq\b$ for all values of $b$, and if the integral
$$\phi(y)=\int_{\a}^{\infty}f(x,y)\, dx$$
converges uniformly with respect to $y$ on $(\a,\b)$, then $\phi(y)$ is a continuous function of $y$ in this interval.  We apply this theorem with $x\to t, y\to x, \a=0, a=0, b>0$. We further assume that $G(x)$ is continuous on $[0,\infty)$ and that the integral converges uniformly with respect to $x$ on some interval $0\leq x\leq\epsilon$.  Hence, using the definition of $J_{\nu}(z)$ from \eqref{sumbesselj}, we find that
\begin{align*}
\lim_{x\to0}\df{F(x)}{x^{\frac{k}{4}-\frac12}}
&=\pi\lim_{x\to0}\int_0^{\infty}G(t)\df{J_{\frac{k}{2}-1}(2\pi\sqrt{xt})}{x^{\frac{k}{4}-\frac12}}\, dt\\
&=\pi\int_0^{\infty}G(t)\lim_{x\to0}\df{J_{\frac{k}{2}-1}(2\pi\sqrt{xt})}{x^{\frac{k}{4}-\frac12}}\, dt\\
&=\pi\int_0^{\infty}G(t)\lim_{x\to0}\left(\df{2\pi\sqrt{xt}}{2}\right)^{\frac{k}{2}-1}\df{dt}{x^{\frac{k}{4}-\frac12}\Gamma(\frac{k}{2})}\\
&=\df{\pi^{\frac{k}{2}}}{\Gamma(\frac{k}{2})}\int_0^{\infty}G(t)t^{\frac{k}{4}-\frac12}\, dt,
\end{align*}
which is identical to \eqref{b1}.

The summation formula \eqref{guirknsuminf} is one among several that Guinand discusses in Section 10 of \cite{guinandconcord}. However, like Popov, Guinand \cite{guinandconcord} does not give conditions for \eqref{guirknsuminf} to hold, for, at the beginning of \cite[p.~263]{guinandconcord}, Guinand says, \textit{``For brevity, no attempt is made to give conditions on $F(x)$ and $G(x)$ for each summation formula''}.

In this paper, using Theorem \ref{guirknsumthm} and imposing further conditions on $F$, we derive the following theorem, thereby rigorously deriving \eqref{guirknsuminf}, and equivalently from the above discussion, \eqref{popres}, for the first time. This is done in Section \ref{guicondition}.

\begin{theorem}\label{guirknsuminfthm}
Let $k$ be a positive integer greater than $3$. Assume that $F$ satisfies the hypotheses of Theorem \textup{\ref{guirknsumthm}}, and that, as $x\to\infty$,
\begin{equation}\label{Fbigo}
F(x)=O_{k}\left(x^{-\frac{k}{4}-\frac{1}{2}-\tau}\right),
\end{equation}
for some fixed $\tau>0$.
Let the function $G$ be defined by
\begin{equation}\label{gee}
G(y)=\pi\int_{0}^{\infty}F(t)J_{\frac{k}{2}-1}(2\pi\sqrt{yt})\, dt,
\end{equation}
and assume that it satisfies
\begin{equation}\label{Gbigo}
G(y)=O_{k}\left(y^{-\frac{k}{4}-\frac{1}{2}-\tau}\right),
\end{equation}
for $\tau>0$, as $y\to\infty$. Then
\begin{align}\label{guirknsuminfallk}
&\sum_{n=1}^{\infty}r_k(n)n^{\frac{1}{2}-\frac{k}{4}}F(n)-\frac{\pi^{\frac{k}{2}}}{\G(\frac{k}{2})}\int_{0}^{\infty}
x^{\frac{k}{4}-\frac{1}{2}}F(x)\, dx\nonumber\\
&=\sum_{n=1}^{\infty}r_k(n)n^{\frac{1}{2}-\frac{k}{4}}G(n)-\frac{\pi^{\frac{k}{2}}}{\G(\frac{k}{2})}\int_{0}^{\infty}
x^{\frac{k}{4}-\frac{1}{2}}G(x)\, dx.
\end{align}
For $k=2$ and $3$, \eqref{guirknsuminfallk} holds if $F$ is continuous on $[0,\infty)$,  $F(x), xF'(x)\in L^{2}(0,\infty)$, and  $F$ satisfies \eqref{Fbigo}, and if $G$ is defined in \eqref{gee} and  satisfies \eqref{Gbigo}.
\end{theorem}
We note that the above theorem is more general than Theorem \ref{vv} with $a(n)=b(n)=r_k(n)$ and $x\to\infty$, since, unlike the latter, it does not require that $F$ be analytic. Moreover, for certain choices of $h$ in this specialized version of Theorem \ref{vv}, it may be very difficult to justify the interchange the order of summation and $\lim_{x\to\infty}$. This turns out to be the case in our choice of the function to prove Theorem \ref{popgenrk} below, that is, \eqref{F}, and hence we resort to Theorem \ref{guirknsuminfthm} for its proof.

If $(a)_n:=a(a+1)\cdots(a+n-1)$, $n\geq1$, and $(a)_0=1$, define the ordinary hypergeometric function $_2F_1$ by
\begin{equation}\label{2f1}
_2F_1\left(\begin{matrix} a,b\\c\end{matrix}\,\Bigr|\, z\right):=\sum_{n=0}^{\infty}\df{(a)_n(b)_n}{(c)_nn!}z^n,\qquad |z|<1.
\end{equation}
We now state in Theorem \ref{popgenrk} a new transformation involving $r_k(n)$, which is the main result of our paper. The main ingredients in the proof of this transformation are Theorem \ref{guirknsuminfthm} and a remarkable generalization of Fock's integral \eqref{fock} given by Koshliakov \cite[eq.~(1)]{kosh34}, namely, for Re$(\mu)>-1$, Re$\left(\mu+2\nu+\tfrac{3}{2}\right)>0$ and Re$(z)>$ Re$(w)>0$,
\begin{align}\label{koshfock}
&\int_{0}^{\infty}J_{\mu}(\rho u)\left(\frac{\sqrt{u^2+z^2}-\sqrt{u^2+w^2}}{\sqrt{u^2+z^2}+\sqrt{u^2+w^2}}\right)^{\nu}\frac{u^{\mu+1}}{\sqrt{u^2+z^2}
\sqrt{u^2+w^2}}\left(\frac{1}{\sqrt{u^2+z^2}}+\frac{1}{\sqrt{u^2+w^2}}\right)^{2\mu}\nonumber\\
&\quad\times\pFq21{\nu-\mu, -\mu}{\nu+1}{\left(\frac{\sqrt{u^2+z^2}-\sqrt{u^2+w^2}}{\sqrt{u^2+z^2}+\sqrt{u^2+w^2}}\right)^{2}}\, du\nonumber\\
&=\frac{\G(\nu+1)}{\G(\nu+\mu+1)}(2\rho)^{\mu}I_{\nu}\left(\frac{\rho(z-w)}{2}\right)K_{\nu}\left(\frac{\rho(z+w)}{2}\right),
\end{align}
where $_2F_1$, $I_{\nu}$, and $K_{\nu}$ are defined by \eqref{2f1}, \eqref{besseli}, and \eqref{besselk}, respectively.
\begin{theorem}\label{popgenrk}
Let $k\geq 2$ be a positive integer. Let $I_{\nu}(z)$ and $K_{\nu}(z)$ denote the modified Bessel functions of the first and second kinds respectively. If $\textup{Re}(\sqrt{\a})\geq\textup{Re}(\sqrt{\b})>0$ and \textup{Re}$(\nu)>0$, then
\begin{align}\label{popgenrkeqn}
&\sum_{n=1}^{\infty}r_k(n)I_{\nu}(\pi(\sqrt{n\a}-\sqrt{n\b}))K_{\nu}(\pi(\sqrt{n\a}+\sqrt{n\b}))\nonumber\\
&=-\frac{1}{2\nu}\left(\frac{\sqrt{\a}-\sqrt{\b}}{\sqrt{\a}+\sqrt{\b}}\right)^{\nu}
+\frac{\Gamma\left(\nu+\frac{k}{2}\right)}{\pi^{\frac{k}{2}}2^{k-1}\Gamma(\nu+1)}
\sum_{n=0}^{\infty}\frac{r_{k}(n)}{\sqrt{n+\a}\sqrt{n+\b}}\left(\frac{\sqrt{n+\a}-\sqrt{n+\b}}{\sqrt{n+\a}+\sqrt{n+\b}}\right)^{\nu}
\nonumber\\
&\quad\times\left(\frac{1}{\sqrt{n+\a}}+\frac{1}{\sqrt{n+\b}}\right)^{k-2}\pFq21{\nu+1-\frac{k}{2}, 1-\frac{k}{2}}{\nu+1}{\left(\frac{\sqrt{n+\a}-\sqrt{n+\b}}{\sqrt{n+\a}+\sqrt{n+\b}}\right)^{2}}.
\end{align}
\end{theorem}
Also applying the inversion formula in the theory of Hankel transforms to \eqref{koshfock}, that is, \eqref{fgrelinf} and \eqref{b2}, we find that the following integral evaluation holds for Re$(\mu)>-1$, Re$\left(\mu+\nu\right)>-1$ and Re$(\pi(z+w))>|$ Re$(\pi(z-w))|+|$ Im$(\rho)|$:
\begin{align*}
&\int_{0}^{\infty}u^{\mu+1}J_{\mu}(\rho u)I_{\nu}(\pi(z-w) u)K_{\nu}(\pi(z+w) u)\, du\\
&=\frac{\G(\nu+\mu+1)}{\G(\nu+1)}\frac{(\rho/2)^{\mu}}{\sqrt{\rho^2+4\pi^2 z^2}\sqrt{\rho^2+4\pi^2 w^2}}\left(\frac{\sqrt{\rho^2+4\pi^2 z^2}-\sqrt{\rho^2+4\pi^2 w^2}}{\sqrt{\rho^2+4\pi^2 z^2}+\sqrt{\rho^2+4\pi^2 w^2}}\right)^{\nu}\nonumber\\
&\quad\times\left(\frac{1}{\sqrt{\rho^2+4\pi^2 z^2}}+\frac{1}{\sqrt{\rho^2+4\pi^2 w^2}}\right)^{2\mu}\pFq21{\nu-\mu, -\mu}{\nu+1}{\left(\frac{\sqrt{\rho^2+4\pi^2 z^2}-\sqrt{\rho^2+4\pi^2 w^2}}{\sqrt{\rho^2+4\pi^2 z^2}+\sqrt{\rho^2+4\pi^2 w^2}}\right)^{2}}.		
\end{align*}
This integral evaluation is the same as in \cite[p.~686, Formula \textbf{6.578.11}]{grn}. To see this, let $c=\rho, a=\pi(z+w)$ and $b=\pi(z-w)$ in the latter formula and simplify.

This paper is organized as follows. In Section \ref{prelim} we collect preliminary results which are used in the sequel. Theorem \ref{guirknsuminfthm} is proved in Section \ref{guicondition}. Section \ref{proofcoro} is devoted to the proof of Theorem \ref{popgenrk} for Re$(\nu)>0$ and to several important corollaries that follow from it. In Section \ref{ac}, we use the principle of analytic continuation to extend the validity of Theorem \ref{popgenrk} to Re$(\nu)>-1$. Of course, it can be analytically continued to Re$(\nu)>-\zeta$ for any $\zeta>0$, however, we refrain ourselves from considering the general case. Finally we conclude our paper with Section \ref{conclusion} consisting of some important remarks and thoughts for further work.
%After we prove this transformation in Section \ref{proofcoro}, we establish the aforementioned results of Hardy, Dixon and Ferrar, and Popov as corollaries.

\section{Preliminary Results}\label{prelim}

	We require several asymptotic formulas.
The asymptotic formulas for the Bessel functions $J_{\nu}(z)$ and $K_{\nu}(z)$, as $|z|\to\infty$, $|\arg(z)|<\pi$, are given by \cite[pp.~199, 202]{watson-1966a}
\begin{align}
J_{\nu}(z)&\sim \left(\frac{2}{\pi z}\right)^{\tf12}\bigg(\cos w\sum_{m=0}^{\infty}\frac{(-1)^m(\nu, 2m)}{(2z)^{2m}} -\sin w\sum_{m=0}^{\infty}\frac{(-1)^m(\nu, 2m+1)}{(2z)^{2m+1}}\bigg),\nonumber\\
K_{\nu}(z)&\sim \left(\frac{\pi}{2 z}\right)^{\tf12}e^{-z}\sum_{m=0}^{\infty}\frac{(\nu, m)}{(2z)^{m}}.\label{kcom1}
\end{align}
Here $w=z-\tfrac{1}{2}\pi \nu-\tfrac{1}{4}\pi$, and
\begin{align*}
(\nu,m)=\frac{\Gamma(\nu+m+1/2)}{\Gamma(m+1)\Gamma(\nu-m+1/2)}.
\end{align*}
From \cite[p.~240, eq.~(9.54)]{temme}, for Re$(z)>0$ and $|z|$ large,
\begin{equation}\label{icom0}
I_{\nu}(z)\sim\frac{e^z}{\sqrt{2\pi z}}\sum_{m=0}^{\infty}(-1)^m\frac{(\nu,m)}{(2z)^m}.
\end{equation}
%where
%\begin{align*}
%(\nu,m)=\frac{\Gamma(\nu+m+1/2)}{\Gamma(m+1)\Gamma(\nu-m+1/2)}.
%\end{align*}
As mentioned in \cite[p.~240]{temme}, \eqref{icom0} is not valid for other complex values of $z$. In fact, in \cite[p.~203]{watson-1966a}, we find that for large values of $|z|$,
\begin{align}\label{icom}
I_{\nu}(z)\sim\frac{e^z}{\sqrt{2\pi z}}\sum_{m=0}^{\infty}(-1)^m\frac{(\nu,m)}{(2z)^m}+\frac{e^{-z\pm\left(\nu+\frac{1}{2}\right)\pi i}}{\sqrt{2\pi z}}\sum_{m=0}^{\infty}\frac{(\nu,m)}{(2z)^m},
\end{align}
where the plus sign in the exponential in the second expression on the right-hand side is taken when $-\frac{1}{2}\pi<\arg(z)<\frac{3}{2}\pi$ and the minus sign is taken when $-\frac{3}{2}\pi<\arg(z)<\frac{1}{2}\pi$.
%The first expression on the right of \eqref{icom} is dominant for Re$(z)>0$ whereas the second one becomes dominant for Re$(z)<0$.
As noted by Watson \cite[p.~203]{watson-1966a}, the discrepancy in the two expansions in \eqref{icom} for  $-\frac{1}{2}\pi<\arg(z)<\frac{1}{2}\pi$ is an example of Stokes' phenomenon and is only an apparent discrepancy. Thus, from \eqref{icom0}, for Re$(z)>0$ and $|z|$ large,
\begin{equation}\label{icom1}
I_{\nu}(z)\sim\frac{e^z}{\sqrt{2\pi z}}.
\end{equation}
We also record below a big-$O$ bound for $r_k(n), k\geq 2$. From \cite[p.~155, Theorem 1]{grosswald} or \cite[eq.~(5)]{rankin}, for $k\geq 5$,
\begin{equation}\label{rknbd5}
r_k(n)=O_{k}(n^{\frac{k}{2}-1}).
\end{equation}
Now Jacobi's four squares theorem \cite[p.~116, Theorem 11.1]{williams} implies that
\begin{equation*}
r_4(n)=8\sigma(n)-32\sigma(n/4),
\end{equation*}
where $\sigma(n)=\sum_{d|n}d$, which along with the elementary fact that $\sigma(n)\leq n(\log n+1)$, implies
\begin{equation}\label{rknbd4}
r_4(n)=O(n\log n).
\end{equation}
One can actually obtain a much better bound, namely $O(n\log\log n)$, by using Robin's inequality \cite[Theorem 2]{robin}.
Jacobi's two squares theorem \cite[p.~79, Theorem 9.3]{williams} is given by
\begin{align*}
r_2(n)=4(d_{1,4}(n)-d_{3,4}(n)),
\end{align*}
where $d_{j,\ell}(n)$ denotes the number of divisors of $n$ congruent to $j\pmod\ell$. Thus,
\begin{align}\label{rknbd2}
r_2(n)=O(d(n))=O(n^{\epsilon})
\end{align}
for any $\epsilon>0$, where the last step follows from \cite[p.~343, Thm.~315]{itn}. It is known \cite[p.~274]{cohen1975}, \cite[p.~134, Thm.~ 8.5]{onoweb} that $r_3(n)$ can be expressed in terms of $H(-n)$, the class number of quadratic forms of discriminant $-n$. If $-n<0$ is a fundamental discriminant, then along with the bound in \cite[Proposition 6.2]{kimrouse}, this implies that
\begin{equation}\label{rknbd3}
r_3(n)=O(n^{1/2}\log n).
\end{equation}
If $-n$ is not a fundamental discriminant, then the formula in \cite[p.~273, Definition 2.2 c)]{cohen1975} with $r=1$ (see also \cite[p.~133, eq.~ (8.1)]{onoweb}) can be used to again obtain \eqref{rknbd3}.

Finally, \eqref{rknbd5}, \eqref{rknbd4}, \eqref{rknbd2} and \eqref{rknbd3} imply that for $k\geq 2$, 
\begin{equation}\label{rknbd}
r_k(n)=O_{k}(n^{\frac{k}{2}-1+\epsilon}),
\end{equation}
for every $\epsilon>0$.

\section{Proof of Theorem \ref{guirknsuminfthm}}\label{guicondition}
We now show that, when $k$ is a positive integer greater than $3$,  \eqref{guirknsuminfallk} follows from Theorem \ref{guirknsumthm}, provided that the functions $F$ and $G$, in addition to the hypotheses of Theorem \ref{guirknsumthm}, respectively satisfy \eqref{Fbigo} and \eqref{Gbigo}. Note that $m\geq 1$ when $k\geq 4$ and so the first hypothesis of Theorem \ref{guirknsumthm} is not vacuous. Also, as will be seen in the proof, \eqref{Fbigo} and \eqref{Gbigo} allow us to let $N\to\infty$ in the associated sums and integrals in \eqref{guirknsum}. However, when $k=2$ or $3$, the first hypothesis in Theorem \ref{guirknsumthm} is vacuous, and we need further conditions given at the end of Theorem \ref{guirknsuminfthm}. These cases are considered at the end of the proof.

Observe first that integrating both sides of \eqref{gee} with respect to $y$ from $0$ to $x$, interchanging the order of integration on the resulting right-hand side, and then employing the differentiation formula
\begin{equation*}
x^{(\nu-1)/2}J_{\nu-1}(2\pi\sqrt{xy})=\frac{1}{\pi\sqrt{y}}\frac{d}{dx}\left(x^{\nu/2}J_{\nu}(2\pi\sqrt{xy})\right),
\end{equation*}
which can readily be deduced from \cite[p.~926, \textbf{8.472.3}]{grn}, we arrive at
\begin{align*}
\int_{0}^{x}G(y)y^{\frac{k}{4}-\frac{1}{2}}\, dy&=\pi\int_{0}^{\infty}F(t)\int_{0}^{x}\frac{1}{\pi\sqrt{t}}\frac{d}{dy}\left(y^{k/4}J_{\frac{k}{2}}\left(2\pi\sqrt{yt}\right)\right)\, dy\, dt\nonumber\\
&=x^{k/4}\int_{0}^{\infty}t^{-1/2}J_{\frac{k}{2}}(2\pi\sqrt{xt})F(t)\, dt,
\end{align*}
thereby obtaining \eqref{fgrel}.

Note that if one of the limits in \eqref{guirknsum} exists, so does the other. To show that
\begin{equation}\label{tanseries}
\lim_{N\to\infty}\sum_{n=1}^{N}\left(1-\frac{n}{N}\right)^{m+1}r_k(n)n^{\frac{1}{2}-\frac{k}{4}}F(n)
=\sum_{n=1}^{\infty}r_k(n)n^{\frac{1}{2}-\frac{k}{4}}F(n),
\end{equation}
we use Tannery's theorem \cite[p.~136]{bromwich} stated below.

\begin{theorem}\label{tannery}
Assume
$\begin{displaystyle}
a_{n} := \lim\limits_{N \to \infty} a_{n}(N)
\end{displaystyle}$ satisfies $|a_{n}(N)|
\leq M_{n}$ with $\begin{displaystyle} \sum_{n=0}^{\infty} M_{n} < \infty
\end{displaystyle}$.
Then
\begin{equation*}\lim \limits_{N \to \infty} \sum_{n=0}^{N} a_{n}(N) =
\sum_{n=0}^{\infty} a_{n}.
\end{equation*}
\end{theorem}

For our application,
\begin{equation*}
a_n(N):=\left(1-\frac{n}{N}\right)^{m+1}r_k(n)n^{\frac{1}{2}-\frac{k}{4}}F(n).
\end{equation*}
Clearly,
\begin{equation*}
\lim_{N\to\infty}a_n(N)=r_k(n)n^{\frac{1}{2}-\frac{k}{4}}F(n).
\end{equation*}
Moreover, from \eqref{rknbd},
\begin{equation*}
r_k(n)\leq L_{k}n^{k/2-1+\epsilon},
\end{equation*}
where $L_{k}$ is a constant depending on $k$ but not on $n$. Also,
\begin{equation*}
\left|1-\frac{n}{N}\right|^{m+1}\leq 2^{m+1}\leq 2^{k/2},
\end{equation*}
since $n\leq N$ and $m=\left\lfloor\frac{1}{2}k\right\rfloor-1$. Hence,
\begin{align*}
|a_n(N)|&\leq 2^{k/2}L_{k}n^{\frac{k}{2}-1+\frac{1}{2}-\frac{k}{4}+\epsilon}|F(n)|.
\end{align*}
Let $M_n:=2^{k/2}L_kn^{\frac{k}{4}-\frac{1}{2}+\epsilon}F(n)$. Now $\sum_{n=0}^{\infty}M_n$ converges if  $F(x)=O_{k}\left(x^{-\frac{k}{4}-\frac{1}{2}-\epsilon-\delta}\right)$ for some $\d>0$. Thus assuming \eqref{Fbigo}, we see that \eqref{tanseries} holds.

Next, in order to show that
\begin{equation}\label{tanint}
\lim_{N\to\infty}\int_{0}^{N}\left(1-\frac{x}{N}\right)^{m+1}x^{\frac{k}{4}-\frac{1}{2}}F(x)\, dx=\int_{0}^{\infty}x^{\frac{k}{4}-\frac{1}{2}}F(x)\, dx,
\end{equation}
we employ the integral analogue of Tannery's theorem \cite[p.~485--486]{bromwich}.

\begin{theorem}\label{intanaltan}
If $\lim_{N\to\infty}u(x, N)=v(x)$ and $\lim_{N\to\infty}\eta_N=\infty$, then
\begin{equation*}
\lim_{N\to\infty}\int_{a}^{\eta_N}u(x, N)\, dx=\int_{a}^{\infty}v(x)\, dx,
\end{equation*}
provided that $u(x, N)$ tends to its limit $v(x)$ uniformly in any fixed interval, and there exists a positive function $T(x)$ such that $|u(x, N)|\leq T(x)$ for all $N$, where $\int_{a}^{\infty}T(x)\, dx$ converges.
\end{theorem}

In our case, $a=0, \eta_N=N$, and
\begin{align*}
u(x, N)&=\left(1-\frac{x}{N}\right)^{m+1}x^{\frac{k}{4}-\frac{1}{2}}F(x),\\
v(x)&=x^{\frac{k}{4}-\frac{1}{2}}F(x).
\end{align*}
Consider any fixed interval, say $[0, B]$. Note that since $F$ is continuous in $[0, B]$, $v(x)$ is bounded on this interval. Also, by the binomial theorem,
\begin{align}
\left(1-\frac{x}{N}\right)^{m+1}&=\sum_{j=0}^{m+1}\binom{m+1}{j}\left(-\frac{x}{N}\right)^{j}\nonumber
=1+O_{m, B}\left(\frac{1}{N}\right),
\end{align}
uniformly for all $x$ in $[0, B]$, as $N\to\infty$. Thus, $u(x, N)\to v(x)$ uniformly in any fixed interval.

Finally,  to apply Theorem \ref{intanaltan}, we need to show that there exists a positive function $T(x)$ such that $|u(x, N)|\leq T(x)$ for all $N$, where $\int_{a}^{\infty}T(x)\, dx$ converges. This is seen at once if we take $T(x)=|v(x)|$ and use \eqref{Fbigo}. Hence, \eqref{tanint} holds if \eqref{Fbigo} holds.

Similarly, \eqref{tanseries} and \eqref{tanint} hold with $F$ replaced by $G$ if \eqref{Gbigo} holds. This completes the proof of Theorem \ref{guirknsuminfthm} when $k\geq 4$.

When $k=2$ or $3$, then $m=0$. If we now examine the hypotheses in Theorem \ref{guirknsumthm}, we notice that the first condition is vacuous. The second condition requires $F$ to be in $L^{2}(0,\infty)$, but this does not necessitate $F$ to be continuous. But if $F$ is not continuous on $[0,\infty)$ and we change the values of $F$ over any set of Lebesgue measure zero, for example, over the set of all positive integers, then even though the right-hand side of \eqref{guirknsuminfallk} as well as the integral on the left side remain the same, the sum on the left side has a different value. However, since any two continuous functions which differ from each other over a set of Lebesgue measure zero must agree everywhere, when $k=2$ and $3$, we require $F$ to be continuous on $[0,\infty)$.

Requiring $F, xF'(x) \in L^{2}(0,\infty)$ ensures that we can perform an integration by parts in \eqref{gee}, so that an argument analogous to that in \cite[Section 2]{guinandsumself2} can be used. This completes the proof for all $k\geq 2$.

\section{Proofs of Theorem \ref{popgenrk} and Its Corollaries}\label{proofcoro}
In this section, we prove our main transformation in Theorem \ref{popgenrk} and then prove numerous corollaries that arise from it.%

\begin{proof}[Theorem \textup{\ref{popgenrk}}][]
Since $I_{\nu}(0)=0$ for Re$(\nu)>0$, it is easy to see that when $\a=\b$, both sides of \eqref{popgenrkeqn} are equal to zero. Hence, \eqref{popgenrkeqn} holds in this case.

For a positive integer $k\geq2$, Re$(\nu)>0$, and Re$(\sqrt{\a})\geq\textup{Re}(\sqrt{\b})>0$, let
\begin{align}\label{eff}
f(x)&:=\frac{1}{\sqrt{x+\a}\sqrt{x+\b}}\left(\frac{\sqrt{x+\a}
-\sqrt{x+\b}}{\sqrt{x+\a}+\sqrt{x+\b}}\right)^{\nu}
\left(\frac{1}{\sqrt{x+\a}}+\frac{1}{\sqrt{x+\b}}\right)^{k-2}\nonumber\\
&\quad\times\pFq21{\nu+1-\frac{k}{2}, 1-\frac{k}{2}}{\nu+1}{\left(\frac{\sqrt{x+\a}-\sqrt{x+\b}}{\sqrt{x+\a}
+\sqrt{x+\b}}\right)^{2}}.
\end{align}
Note that $f$ is continuous on $[0,\infty)$. We prove \eqref{popgenrkeqn} first for $\a>\b>0$ and $k\geq 2$. After this, we prove \eqref{popgenrkeqn} in its full generality, that is, for Re$(\sqrt{\a})\geq$ Re$(\sqrt{\b})>0$, by applying the principle of analytic continuation.

%For $n\geq 1$, let $\l_n=\mu_n=\pi n$ and  $a(n)=b(n)=r_k(n)$ in Theorem \ref{vv}. From \cite[p.~18]{cn}, it is known that if
%$$ \zeta_k(s):=\sum_{n=1}^{\infty}\df{r_k(n)}{n^s}, \qquad k\geq 2, \quad \text{Re }s>\tf12k,$$
%then $\zeta_k(s)$ can be analytically continued to the entire complex plane, except for a simple pole at $s=\tf12 k$ with residue $\pi^{k/2}/\Gamma(\tf12k)$,  and moreover $\zeta_k(s)$ satisfies the functional equation
%\begin{equation}\label{funcequa1}
%\pi^{-s}\Gamma(s)\zeta_k(s)=\pi^{-(k/2-s)}\Gamma(\tf12k-s)\zeta_k(\tf12k-s).
%\end{equation}
%It is then clear that $r=\tf12k$. Also note that $\phi(0)=\psi(0)=-1$. Hence, after replacing $x$ by $\pi x$ in \eqref{voronoib}, we have
%\begin{align}\label{0.2}
%\sum_{n\leq x}r_k(n)f(\pi n)=-\lim_{a\to 0}f(a)+\int_{0}^{\pi x}Q_{0}'(t)f(t)\, dt+\sum_{n=1}^{\infty}\frac{r_k(n)}{(\pi n)^{\frac{k}{2}-1}}\int_{0}^{\pi x}\mathcal{I}_{-1}(\pi nt)f(t)\, dt.
%\end{align}
%Note that
%\begin{equation}
%Q_{0}(t)=-1+\frac{t^{\frac{k}{2}}}{\G(\frac{k}{2}+1)}.
%\end{equation}
%Thus,
%\begin{align}\label{befcov}
%\sum_{n\leq x}r_k(n)f(\pi n)&=-\lim_{a\to 0}f(a)+\frac{\pi^{\frac{k}{2}}}{\G(\frac{k}{2})}\int_{0}^{x}y^{\frac{k}{2}-1}f(\pi y)\, dy\nonumber\\
%&\quad+\pi\sum_{n=1}^{\infty}\frac{r_k(n)}{(\pi n)^{\frac{k}{2}-1}}\int_{0}^{x}(\pi^2 ny)^{\frac{k}{4}-\frac{1}{2}}J_{\frac{k}{2}-1}(2\pi\sqrt{ny})f(\pi y)\, dy,
%\end{align}
%by a simple change of variable $t=\pi y$ in the two integrals in \eqref{0.2}.
Let
\begin{equation}\label{F}
F(x)=x^{\frac{k}{4}-\frac{1}{2}}f(x),
\end{equation}
where $f$ is defined in \eqref{eff}. Then $F$ is continuous on $[0,\infty)$.
Let $u=\sqrt{y}, \mu=\frac{k}{2}-1, z^2=\a, w^2=\b$ and $\rho=2\pi\sqrt{x}$
in \eqref{koshfock}, so that by \eqref{gee}, for Re$(\sqrt{\a})> $Re$(\sqrt{\b})>0$
and Re$(\nu)>-\frac14(k+1)$,
\begin{align}\label{koshfockspl}
G(x)&=\pi\int_{0}^{\infty}y^{\frac{k}{4}-\frac{1}{2}}f(y)J_{\frac{k}{2}-1}(2\pi\sqrt{xy})\, dy\nonumber\\
&=2^{k-1}\pi^{\frac{k}{2}}x^{\frac{k}{4}-\frac{1}{2}}\frac{\G(\nu+1)}{\G\left(\nu+\frac{k}{2}\right)}
I_{\nu}(\pi(\sqrt{x\a}-\sqrt{x\b}))K_{\nu}(\pi(\sqrt{x\a}+\sqrt{x\b})).
\end{align}

We apply Theorem \ref{guirknsuminfthm} with $F$ and $G$ as given above. To do this, we first show that they satisfy the hypotheses of Theorem \ref{guirknsuminfthm}.
Since $f$ is analytic, and hence infinitely differentiable, $F(x), F'(x), F''(x),\dots,
F^{(2m-1)}(x)$ are integrals. Also, as $x\to\infty$,
\begin{equation}\label{fas}
f(x)\sim\frac{2^{k-2-2\nu}(\a-\b)^{\nu}}{x^{\nu+k/2}},
\end{equation}
since the hypergeometric function tends to $1$ as $x\to\infty$. Since Re$(\nu)>0$, from \eqref{F} and \eqref{fas}, it is easy to see that the bound in \eqref{Fbigo} holds.
From \eqref{icom1}, for Re$(\sqrt{\a})>$ Re$(\sqrt{\b})$ and large $n$,
\begin{equation}\label{icomu}
I_{\nu}(\pi(\sqrt{n\a}-\sqrt{n\b})\sim\frac{e^{\pi(\sqrt{n\a}-\sqrt{n\b})}}{\pi\sqrt{2(\sqrt{n\a}-\sqrt{n\b})}}.
\end{equation}
Also from \eqref{kcom1}, for $(\sqrt{\a}+\sqrt{\b})\in\mathbb{C}\backslash(-\infty,0)$ and large $n$,
{\allowdisplaybreaks\begin{align}\label{kcom}
K_{\nu}(\pi(\sqrt{n\a}+\sqrt{n\b}))\sim\frac{e^{-\pi(\sqrt{n\a}+\sqrt{n\b})}}{\sqrt{2(\sqrt{n\a}+\sqrt{n\b})}}.
\end{align}}%
Thus, since $\a\geq\b>0$, by \eqref{icomu} and \eqref{kcom}, the function $G$ defined in \eqref{koshfockspl} decays exponentially as $x\to\infty$. Therefore, $G$ satisfies \eqref{Gbigo}.

Next, we show that, for $0\leq r\leq 2m$, $x^rF^{(r)}(x)\in L^{2}(0,\infty)$.
It is easy to see that $f$ (and hence $F$) is analytic in a fixed small disk centered at the origin. Thus, for any non-negative integer $r$, $f^{(r)}$  (and hence $F^{(r)}$) is bounded in this disk. If we restrict $x$ to be non-negative and real, then the fact that $f^{(r)}(x)$ is continuous on $[0,\epsilon_0]$,  for some $\epsilon_0>0$, implies that $x^rf^{(r)}(x)$ (and hence $x^rF^{(r)}(x)$) belongs to $L^{2}([0,\epsilon_0])$. In fact, the continuity of $x^rf^{(r)}(x)$ on $[0,\infty)$ implies that $x^rf^{(r)}(x)$ (and hence $x^rF^{(r)}(x)$) belongs to $\in L^{2}([\epsilon_0,A])$, where $A$ is any large fixed real number. It remains to show that all functions of the form $x^rF^{(r)}(x)\in L^{2}([A,\infty))$. This is done next.

Note that $f$ (and hence $F$) is analytic at $\infty$, as can be seen from \eqref{eff},  the fact that Re$(\nu)>0$, and that the hypergeometric function tends to $1$ as $x\to\infty$. Also, $A$ lies inside a neighborhood of $\infty$. Thus, $F$ can be expanded as a power series in $1/x$, and in fact,
\begin{equation*}
F(x)=\sum_{\ell=0}^{\infty}\frac{c_{\ell}}{x^{\ell+\nu+\frac{k}{4}+\frac{1}{2}}},
\end{equation*}
where the constants $c_{\ell}$, $\ell\geq0$, depend on $\nu,k, \a$, and $\b$, and $c_0=2^{k-2-2\nu}(\a-\b)^{\nu}$, which can be seen from \eqref{fas}. Hence, for  $r\geq0$,
\begin{equation*}
x^rF^{(r)}(x)=\sum_{\ell=0}^{\infty}\frac{d_{\ell}}{x^{\ell+\nu+\frac{k}{4}+\frac{1}{2}}},
\end{equation*}
where the coefficients $d_{\ell}$, $\ell\geq0$, depend on $r, \nu,k, \a$, and $\b$. Hence,
\begin{equation*}
x^rF^{(r)}(x)=O_{r, \nu, k, \a, \b}\left(x^{-\nu-\frac{k}{4}-\frac{1}{2}}\right)
\end{equation*}
on the interval $[A, \infty)$. Since Re$(\nu+\frac{k}{4}+\frac{1}{2})>\frac12$, we see that $x^rF^{(r)}(x)\in L^{2}([A,\infty))$.

We have thus shown that $x^rF^{(r)}(x)\in L^{2}((0,\infty))$ for any $r$, in particular, for $0\leq r\leq 2m$. Also, when $k=2, 3$, the additional condition $xF'(x)\in L^{2}((0,\infty))$ is similarly seen to be true.
%The above analysis implies that $x^rf^{(r)}(x)$ (and hence $x^rF^{(r)}(x)$) belongs to $L^{2}([0,\infty))$ for any $r$ such that $0\leq r\leq 2m$.
Thus, the hypotheses of Theorem \ref{guirknsuminfthm} are satisfied.
%As explained in the previous section, if $F(x)=O_{k}\left(x^{-\frac{k}{4}-\frac{1}{2}-\tau}\right)$ for any $\tau>0$ as $x\to\infty$, or equivalently, if $f(\pi x)=O_{k}\left(x^{-\frac{k}{2}-\tau}\right)$, then both \eqref{tanseries} and \eqref{tanint} hold. The bound for $f$ indeed holds as can be seen from \cite[eqn. (3.7)]{bdkz2}, since Re$(\nu)>0$.
%It is easy to see that because of our initial assumption $\a\geq\b>0$ in the proof of our main transformation in the paper, the exponential decay, as $x\to\infty$, of the product of two Bessel functions in the function $G$ defined in \eqref{g} (see \cite[eqns. (3.21), (3.22)]{bdkz2}) guarantees that $G$ satisfies \eqref{gbou}. Hence, the analogues of \eqref{tanseries} and \eqref{tanint} hold for $G$ as well.
%The complete analysis in this section finally proves that \eqref{guirknsuminf} holds for $k>3$ with \eqref{F} and with $f$ defined in \eqref{eff}.

Hence, from Theorem \ref{guirknsuminfthm}, \eqref{eff}, \eqref{F}, and \eqref{koshfockspl}, we find that
{\allowdisplaybreaks\begin{align}\label{befcov1}
&\sum_{n=1}^{\infty}\frac{r_k(n)}{\sqrt{n+\a}\sqrt{n+\b}}\left(\frac{\sqrt{n+\a}
-\sqrt{n+\b}}{\sqrt{n+\a}+\sqrt{n+\b}}\right)^{\nu}
\left(\frac{1}{\sqrt{n+\a}}+\frac{1}{\sqrt{n+\b}}\right)^{k-2}\nonumber\\
&\quad\times\pFq21{\nu+1-\frac{k}{2}, 1-\frac{k}{2}}{\nu+1}{\left(\frac{\sqrt{n+\a}-\sqrt{n+\b}}{\sqrt{n+\a}
+\sqrt{n+\b}}\right)^{2}}-\frac{\pi^{\frac{k}{2}}}{\G(\frac{k}{2})}\int_{0}^{\infty}x^{\frac{k}{2}-1}f(x)\, dx\nonumber\\
&=2^{k-1}\pi^{k/2}\frac{\G(\nu+1)}{\G(\nu+\frac{k}{2})}\sum_{n=1}^{\infty}r_k(n)I_{\nu}(\pi(\sqrt{n\a}-\sqrt{n\b}))
K_{\nu}(\pi(\sqrt{n\a}+\sqrt{n\b}))\nonumber\\
&\quad-\frac{\pi^{\frac{k}{2}}}{\G(\frac{k}{2})}\int_{0}^{\infty}x^{\frac{k}{4}-\frac{1}{2}}G(x)\, dx.
\end{align}}
% We need to justify that letting $x\to\infty$ in \eqref{befcov} indeed does yield \eqref{befcov1}. Consider the first integral on the right-hand side of \eqref{befcov1}. The integrand is of course well behaved at the lower limit $0$.  Noting that the hypergeometric function approaches $1$ as $y\to\infty$, we find that, for large $y$,

 %Hence, Re$(\nu)>0$ is required to secure the convergence of the first integral to the right of the equality sign in \eqref{befcov1}. The asymptotic formula \eqref{fas} also implies that the series on the left-hand side of \eqref{befcov1} converges for Re$(\nu)>0$, since we know that $\sum_{n=1}^{\infty}r_k(n)n^{-s}$ converges for Re$(s)>\tf12k$.

%When we let $x\to\infty$ in the series integrals on the right-hand side of \eqref{befcov}, we obtain a special case of an integral considered by Koshliakov, namely, \eqref{koshfockspl} below. The integrand is, of course, well-behaved near $0$. By \eqref{asymbess}, as $y\to\infty$,
%\begin{equation*}
%y^{\frac{k}{4}-\frac{1}{2}}J_{\frac{k}{2}-1}(2\pi\sqrt{ny})f(\pi y)\sim\frac{2^{k-2-2\nu}(\a-\b)^{\nu}}{\pi n^{\frac{1}{4}}y^{\nu+\frac{k+3}{4}}}\cos\left(2\pi\sqrt{ny}-\frac{1}{4}\pi(k-1)\right).
%\end{equation*}
%Thus, we see that the aforementioned series integrals converge at $\infty$, provided that Re$(\nu)>-\frac{1}{4}(k+1)$.

We now evaluate the two integrals occurring in \eqref{befcov1}.
For Re$(\mu+\nu+1)>0$, Re$(a)>0$, and $|b|<|a|$ \cite[p.~385, eq.~(3)]{watson-1966a},
\begin{equation}\label{watbes}
\int_{0}^{\infty}e^{-at}J_{\nu}(bt)t^{\mu-1}\, dt=\frac{(b/(2a))^{\nu}\G(\mu+\nu)}{a^{\mu}\G(\nu+1)}\left(1+\frac{b^2}{a^2}\right)^{\tfrac{1}{2}-\mu}
\pFq21{\frac{\nu-\mu+1}{2},\frac{\nu-\mu}{2}+1}{\nu+1}{-\frac{b^2}{a^2}}.
\end{equation}
 However, as given in \cite[p.~385]{watson-1966a}, by analytic continuation in $b$, \eqref{watbes} is valid for Re$(a\pm ib)>0$. We would like to substitute $b=i$ and $a=1/u$ in \eqref{watbes}, where
\begin{align}\label{u}
u=\frac{\a-\b}{2y+\a+\b},\hspace{5mm}y\geq 0.
\end{align}
 Since $\a>\b>0$, the condition $\text{Re}(a)=\text{Re}(1/u)>1$ is obviously satisfied. We now use,  from \eqref{besseli}, the relation $I_{\nu}(t)=i^{-\nu}J_{\nu}(it)$,
  %where $I_{\nu}(t)$ is the modified Bessel function of the first kind,
   to deduce that, for Re$(\mu+\nu+1)>0$, \footnote{This integral evaluation is also given in \cite{kosh34}, but with less details.}
\begin{align}\label{watkosh}
\int_{0}^{\infty}e^{-t/u}I_{\nu}(t)t^{\mu}\, dt=\frac{u^{\nu+\mu+1}\G(\nu+\mu+1)}{2^{\nu}\G(\nu+1)(1-u^2)^{\mu+1/2}}\cdot\pFq21{\frac{\nu-\mu}{2}, \frac{\nu-\mu+1}{2}}{\nu+1}{u^2}.
\end{align}
Also from \cite[p.~391]{olver-2010a}, for $|z|<1$,
\begin{equation}\label{hyptra0}
\pFq21{a,b}{a-b+1}{z}=(1+z)^{-a}\pFq21{\frac{1}{2}a,\frac{1}{2}(a+1)}{a-b+1}{\frac{4z}{(1+z)^2}}.
\end{equation}
Since $u<1$, we see that $\left|(1-\sqrt{1-u^2})/(1+\sqrt{1-u^2})\right|<1$. Hence, let  $a=\nu-\mu$, $b=-\mu$, and  $z=(1-\sqrt{1-u^2})/(1+\sqrt{1-u^2})$ in \eqref{hyptra0} to find that
\begin{equation}\label{hyptra}
\pFq21{\frac{\nu-\mu}{2}, \frac{\nu-\mu+1}{2}}{\nu+1}{u^2}=\frac{2^{\nu-\mu}}{(1+\sqrt{1-u^2})^{\nu-\mu}}\pFq21{\nu-\mu, -\mu}{\nu+1}{\frac{1-\sqrt{1-u^2}}{1+\sqrt{1-u^2}}}.
\end{equation}
Substituting \eqref{hyptra} in \eqref{watkosh}, employing the representation \eqref{u} of $u$ in terms of $y, \a$ and $\b$, and using the elementary identities
$$(2y+\a+\b)^2-(\a-\b)^2=4(y+\a)(y+\b)$$
and
$$ 2y+\a+\b\pm2\sqrt{(y+\a)(y+\b)}=(\sqrt{y+\a}\pm\sqrt{y+\b})^2$$
in our simplification, we arrive at
\begin{align}\label{watkosh1}
\int_{0}^{\infty}e^{-t/u}I_{\nu}(t)t^{\mu}\, dt&=\frac{2^{-3\mu-1}}{\sqrt{y+\a}\sqrt{y+\b}}\frac{(\a-\b)^{\nu+\mu+1}\G(\nu+\mu+1)}{(\sqrt{y+\a}+\sqrt{y+\b})^{2\nu}\G(\nu+1)}
\left(\frac{1}{\sqrt{y+\a}}+\frac{1}{\sqrt{y+\b}}\right)^{2\mu}\nonumber\\
&\quad\times\pFq21{\nu-\mu, -\mu}{\nu+1}{\left(\frac{\sqrt{y+\a}-\sqrt{y+\b}}{\sqrt{y+\a}+\sqrt{y+\b}}\right)^2}.
\end{align}
Now we let $\mu=\frac{k}{2}-1$. With the choice of $f$ in \eqref{eff}, the use of \eqref{watkosh1}, the simple algebraic identity
$$ \left(\df{\sqrt{y+\a}-\sqrt{y+\b}}{\sqrt{y+\a}+\sqrt{y+\b}}\right)^{\nu}=\df{(\a-\b)^{\nu}}{(\sqrt{y+\a}+\sqrt{y+\b})^{2\nu}},$$ and an inversion in the order of integration, we deduce that,
 for Re$(\nu+\frac{k}{2})>0$,
\begin{align}\label{choi1}
\int_{0}^{\infty}y^{\frac{k}{2}-1}f(y)\, dy&=\frac{2^{\frac{3k}{2}-2}\G(\nu+1)}{(\a-\b)^{\frac{k}{2}}\G\left(\nu+\frac{k}{2}\right)}
\int_{0}^{\infty}y^{\frac{k}{2}-1}\int_{0}^{\infty}e^{-(2y+\a+\b)t/(\a-\b)}t^{\frac{k}{2}-1}I_{\nu}(t)\, dt\, dy\nonumber\\
&=\frac{2^{\frac{3k}{2}-2}\G(\nu+1)}{(\a-\b)^{\frac{k}{2}}\G\left(\nu+\frac{k}{2}\right)}
\int_{0}^{\infty}t^{\frac{k}{2}-1}e^{-\left(\frac{\a+\b}{\a-\b}\right)t}I_{\nu}(t)\, dt\int_{0}^{\infty}e^{-\frac{2yt}{(\a-\b)}}y^{\frac{k}{2}-1}\, dy\nonumber\\
&=\frac{2^{k-2}\G(\nu+1)\G\left(\frac{k}{2}\right)}{\G\left(\nu+\frac{k}{2}\right)}
\int_{0}^{\infty}e^{-\left(\frac{\a+\b}{\a-\b}\right)t}I_{\nu}(t)\, \frac{dt}{t},
\end{align}
where the inversion in order of integration can be easily justified with the use of \eqref{icom}, and where in the last step we used the integral representation for the gamma function, namely, for Re$(z)>0$, Re$(\l)>0$,
\begin{equation*}
\int_{0}^{\infty}e^{-\l t}t^{z-1}\, dt=\l^{-z}\G(z).
\end{equation*}
%[JUSTIFY THE INTERCHANGE OF THE ORDER OF INTEGRATION.]% I think that we do not need to do this because the exponentials, after applying the asymptotic formula for I_{nu}, have negative exponents\\

%From \cite[p.~703, eq.~\textbf{6.624.5}]{grn}, for Re$(\nu+\mu)>-1$,
%\begin{equation}\label{eint1}
%\int_{0}^{\infty}\textup{exp}\left(\frac{-tz}{\sqrt{z^2-1}}\right)I_{\mu}(t)t^{\nu}\,
%dt=\G(\nu+\mu+1)(z^2-1)^{\frac{\nu+1}{2}}P_{\nu}^{-\mu}(z),
%\end{equation}

%Now let $\nu=-1$ and $z=(\a+\b)/(2\sqrt{\a\b})$, and then replace $\mu$ by $\nu$ in
%\eqref{eint1} and \eqref{eint2}, so that for Re$(\nu)>0$,
%\footnote{We note that one could also evaluate this integral by letting
%  $\mu=-1$ in \eqref{watkosh1}. However, one would need to justify that
 % \eqref{watkosh1} can be used with $y=0$.}
Now let $y=0$ and $\mu=-1$ in \eqref{watkosh1} to see that
\begin{align}\label{choi2}
\int_{0}^{\infty}e^{-\left(\frac{\a+\b}{\a-\b}\right)t}I_{\nu}(t)\,
\frac{dt}{t} =\frac{1}{\nu}\left(\frac{\sqrt{\a}-\sqrt{\b}}{\sqrt{\a}+\sqrt{\b}}\right)^{\nu}.
\end{align}
Thus, from \eqref{choi1} and \eqref{choi2}, for Re$(\nu)>0$,
\begin{align}\label{choi3}
\int_{0}^{\infty}y^{\frac{k}{2}-1}f(y)\, dy=\frac{2^{k-2}\G(\nu)\G\left(\frac{k}{2}\right)}{\G\left(\nu+\frac{k}{2}\right)}
\left(\frac{\sqrt{\a}-\sqrt{\b}}{\sqrt{\a}+\sqrt{\b}}\right)^{\nu}.
\end{align}

(Observe that if we set $k=2$ in \eqref{eff}, the
hypergeometric function therein reduces to 1. Therefore, a more elementary
proof of \eqref{choi3} can be established by the change of variable
$$t=-\log\left(\frac{\sqrt{y+\a}-\sqrt{y+\b}}{\sqrt{y+\a}+\sqrt{y+\b}}\right).$$
We leave the details to the reader.)

%Let $u=\sqrt{y}, \mu=\frac{k}{2}-1, z^2=\a, w^2=\b$ and $\rho=2\pi\sqrt{n}$
%in \eqref{koshfock} so as to obtain, for Re$(\sqrt{\a})> $Re$(\sqrt{\b})>0$
%and Re$(\nu)>-\frac14(k+1)$,
%\begin{align}\label{koshfockspl}
%&\int_{0}^{\infty}y^{\frac{k}{4}-\frac{1}{2}}J_{\frac{k}{2}-1}(2\pi\sqrt{ny})f(\pi y)\, dy\nonumber\\
%&=2^{k-1}\pi^{\frac{k}{2}-1}n^{\frac{k}{4}-\frac{1}{2}}\frac{\G(\nu+1)}{\G\left(\nu+\frac{k}{2}\right)}
%I_{\nu}(\pi(\sqrt{n\a}-\sqrt{n\b}))K_{\nu}(\pi(\sqrt{n\a}+\sqrt{n\b})).
%\end{align}

It remains to evaluate the integral on the right-hand side of \eqref{befcov1}. To that end,
we use the formula from \cite[p.~380--381, eq.~\textbf{2.16.28.1}]{pbm}. For $|\textup{Re}$ $b|<$ Re $c$, and $|\textup{Re}$ $\nu|<$ Re $(a+\nu)$,
\begin{align}\label{bc}
\int_{0}^{\infty}x^{a-1}I_{\nu}(bx)K_{\nu}(cx)\, dx=2^{a-2}(c^2-b^2)^{-a/2}\G\left(\frac{a}{2}\right)\G\left(\nu+\frac{a}{2}\right)P_{-a/2}^{-\nu}\left(\frac{c^2+b^2}{c^2-b^2}\right),
\end{align}
where $P_{\nu}^{\mu}(z)$ is the associated Legendre function of the first kind defined by \cite[p.~959, eq.~\textbf{8.702}]{grn}
\begin{equation}\label{eint2}
P_{\nu}^{\mu}(z):=\frac{1}{\G(1-\mu)}\left(\frac{z+1}{z-1}\right)^{\mu/2}\pFq21{-\nu, \nu+1}{1-\mu}{\frac{1-z}{2}}.
\end{equation}
Let $a=k$, $b=\pi(\sqrt{\a}-\sqrt{\b})$ and $c=\pi(\sqrt{\a}+\sqrt{\b})$ in \eqref{bc}. The conditions for the validity of \eqref{bc} are easily seen to hold. Employing \eqref{eint2} followed by an application of Pfaff's formula \cite[p.~390]{olver-2010a}, namely,
\begin{equation*}
\pFq21{a, b}{c}{z}=(1-z)^{-b}\pFq21{c-a, b}{c}{\frac{z}{z-1}},
\end{equation*}
%and
%\begin{equation}
%\pFq21{a, b}{c}{z}=(1-z)^{c-a-b}\pFq21{c-a, c-b}{c}{z},
%\end{equation}
we obtain, upon simplification,
\begin{align}\label{choi4}
&-\frac{\pi^{\frac{k}{2}}}{\G(\frac{k}{2})}\int_{0}^{\infty}x^{\frac{k}{4}-\frac{1}{2}}G(x)\, dx\nonumber\\
&=-\frac{1}{\sqrt{\a\b}}\left(\frac{\sqrt{\a}-\sqrt{\b}}{\sqrt{\a}+\sqrt{\b}}\right)^{\nu}
\left(\frac{1}{\sqrt{\a}}+\frac{1}{\sqrt{\b}}\right)^{k-2}\pFq21{\nu+1-\frac{k}{2}, 1-\frac{k}{2}}{\nu+1}{\left(\frac{\sqrt{\a}-\sqrt{\b}}{\sqrt{\a}
+\sqrt{\b}}\right)^{2}},
\end{align}
which, indeed, agrees with \eqref{b1}, as can be seen from \eqref{F}.

Substituting \eqref{choi3} and \eqref{choi4} into \eqref{befcov1}, we find that, for Re$(\nu)>0$,
\begin{align}\label{3.15}
&\sum_{n=0}^{\infty}\frac{r_{k}(n)}{\sqrt{n+\a}\sqrt{n+\b}}\left(\frac{\sqrt{n+\a}
-\sqrt{n+\b}}{\sqrt{n+\a}+\sqrt{n+\b}}\right)^{\nu}\nonumber\\
&\quad\times\left(\frac{1}{\sqrt{n+\a}}+\frac{1}{\sqrt{n+\b}}\right)^{k-2}
\pFq21{\nu+1-\frac{k}{2}, 1-\frac{k}{2}}{\nu+1}{\left(\frac{\sqrt{n+\a}
-\sqrt{n+\b}}{\sqrt{n+\a}+\sqrt{n+\b}}\right)^{2}}\nonumber\\
&=\frac{2^{k-2}\pi^{\frac{k}{2}}\G(\nu)}{\G\left(\nu+\frac{k}{2}\right)}
\left(\frac{\sqrt{\a}-\sqrt{\b}}{\sqrt{\a}+\sqrt{\b}}\right)^{\nu}\nonumber\\
&\quad+\frac{2^{k-1}\pi^{\frac{k}{2}}\G(\nu+1)}{\G\left(\nu+\frac{k}{2}\right)}
\sum_{n=1}^{\infty}r_k(n)I_{\nu}(\pi(\sqrt{n\a}-\sqrt{n\b}))K_{\nu}(\pi(\sqrt{n\a}+\sqrt{n\b})),
\end{align}%
which easily simplifies to \eqref{popgenrkeqn}. This completes the proof of \eqref{popgenrkeqn} for $\a\geq\b>0$ and $k\geq 2$.

By \eqref{icomu} and \eqref{kcom},
the series on the right-hand side of \eqref{3.15}
%the series
%$\sum_{n=1}^{\infty}r_k(n)I_{\nu}(\pi(\sqrt{n\a}-\sqrt{n\b}))K_{\nu}(\pi(\sqrt{n\a}+\sqrt{n\b}))$
converges absolutely for Re$(\sqrt{\a})>$ Re$(\sqrt{\b})>0$. Note that if
Re$(\sqrt{\a})=$ Re$(\sqrt{\b})$,
then from \eqref{icom}, it is seen that
$$I_{\nu}\left(i\pi\sqrt{n}\left(\text{Im}(\sqrt{\a})
-\text{Im}(\sqrt{\b})\right)\right)=O_{\a,\b,\nu}(n^{-1/4}).$$
This, together with  \eqref{kcom},
implies that the series on the right-hand side of \eqref{3.15} again converges absolutely for Re$(\sqrt{\a})=$
Re$(\sqrt{\b})>0$. Using \eqref{eff} and \eqref{fas}, we see that the
right-hand side of \eqref{popgenrkeqn} is well-defined for
Re$(\sqrt{\a})\geq$ Re$(\sqrt{\b})>0$.

Fix any real positive $\b_0$, let $\Omega_{\b_{0}}$ represent the region in
the $\a$-complex plane given by  Re$(\sqrt{\a})>\sqrt{\b_0}>0$. The region
$\Omega_{\b_{0}}$ has, as its boundary, the parabola given by the equation
Re$(\a)=\b_0-\frac{1}{4\b_0}\textup{Im}(\a)^2$. Both sides of
\eqref{popgenrkeqn}, with $\b$ replaced by $\b_0$, are analytic in
$\Omega_{\b_{0}}$. Since $\Omega_{\b_{0}}$ is simply connected and since by
what we have proved so far, the two sides of \eqref{popgenrkeqn} coincide on
$\Omega_{\b_{0}}\cap\mathbb{R}$, it follows that \eqref{popgenrkeqn} holds
for all $\a\in\Omega_{\b_{0}}$.

Now fix any two complex numbers $\a_1$ and $\b_1$ satisfying
Re$(\sqrt{\a_1})>$ Re$(\sqrt{\b_1})>0$. Fix an arbitrary real positive number
$\b^{*}$ such that $\sqrt{\b^{*}}\leq$ Re$(\sqrt{\b_1})$. Note that $\a_1$
belongs to $\Omega_{\b^{*}}$ for every such $\b^{*}$. Therefore
\eqref{popgenrkeqn}, with $\a$ replaced by $\a_1$, holds with $\b$ replaced
by any such $\b^{*}$. Next, let $\mathfrak{D}_{\a_1}$ be the region in the
$\b$-complex plane given by Re$(\sqrt{\a_1})>$ Re$(\sqrt{\b})>0$. Here
$\mathfrak{D}_{\a_1}$ is simply connected with its boundary formed by the
parabola given by the equation Re$(\b)=$
Re$(\sqrt{\a_1})^2-\frac{1}{4\textup{Re}(\sqrt{\a_1})^2}$ Im$(\b)^2$ and the
negative real line. Both sides of \eqref{popgenrkeqn} are analytic on
$\mathfrak{D}_{\a_1}$ and coincide at all points $\b=\b^{*}$, with $\b^{*}$
as above, and hence they coincide everywhere inside $\mathfrak{D}_{\a_1}$.

Since $\b_1$ belongs to $\mathfrak{D}_{\a_1}$, it follows that
\eqref{popgenrkeqn} holds for $\a=\a_1$ and $\b=\b_1$, as desired. Lastly, by
continuity, \eqref{popgenrkeqn} also holds for all those complex numbers $\a$
and $\b$ for which Re$(\sqrt{\a})=$ Re$(\sqrt{\b})>0$. This completes the
proof of Theorem \ref{popgenrk}.
\end{proof}

%If we fix any $\b=\b_0$ such that Re$(\sqrt{\b_0})>0$, then
%Re$(\sqrt{\a})\geq$ Re$(\sqrt{\b_0})>0$ represents a simply connected region
%of the $\a$-complex plane. Hence the series represents an analytic function
%of $\a$ in this region. Now if we fix any $\a=\a_0$ and look at the vertical
%strip $0<$Re$(\sqrt{\b})<$Re$(\sqrt{\a_0})$, then the fact that again this
%region is also simply connected implies that the series represents an
%analytic function there. Thus it is analytic in $\a$ and $\b$ for the region
%Re$(\sqrt{\a})>$ Re$(\sqrt{\b})>0$. Since the right-hand side of
%\eqref{popgenrkeqn} too is analytic in this region, by the principle of
%analytic continuation, \eqref{popgenrkeqn} holds for Re$(\sqrt{\a})>$
%Re$(\sqrt{\b})>0$. Finally, the result also holds for Re$(\sqrt{\a})=$
%Re$(\sqrt{\b})>0$ by continuity.

%\noindent\textbf{Remark 2.}
\begin{remark} Let $\l,\d>0$, \textup{Re}$(\nu)>0$ and \textup{Re}$(z)>0$,
%\begin{align}
%\lim_{z\to 0}I_{\nu}(\l z)K_{\nu}(\d z)=\frac{1}{2\nu}\left(\frac{\l}{\d}\right)^{\nu}.
%\end{align}
 Using \eqref{besselk} and \eqref{besseli}, we find that, for \textup{Re}$(\nu)>0$,
\begin{align*}
\lim_{z\to 0}I_{\nu}(\l z)K_{\nu}(\d z)&=\frac{\pi}{2\sin\nu\pi}\lim_{z\to
  0}\left(I_{\nu}(\l z)I_{-\nu}(\d z)-I_{\nu}(\l z)I_{\nu}(\d z
  )\right)\nonumber\\
&=\frac{\pi}{2\sin\nu\pi}\left(\frac{(\l/\d)^{\nu}}{\G(1+\nu)\G(1-\nu)}-0\right)\nonumber\\
&=\frac{1}{2\nu}\left(\frac{\l}{\d}\right)^{\nu},
\end{align*}
by using the reflection formula for the gamma function. Thus,
\begin{equation}\label{n0lim}
\lim_{x\to 0}I_{\nu}(\pi(\sqrt{x\a}-\sqrt{x\b}))K_{\nu}(\pi(\sqrt{x\a}+\sqrt{x\b}))
=\frac{1}{2\nu}\left(\frac{\sqrt{\a}-\sqrt{\b}}{\sqrt{\a}+\sqrt{\b}}\right)^{\nu},
\end{equation}
and so the identity in Theorem \ref{popgenrk} can also be written in the more compact form
\begin{align*}
&\sum_{n=0}^{\infty}r_k(n)I_{\nu}(\pi(\sqrt{n\a}-\sqrt{n\b}))K_{\nu}(\pi(\sqrt{n\a}+\sqrt{n\b}))\nonumber\\
&=\frac{\Gamma\left(\nu+\frac{k}{2}\right)}{\pi^{\frac{k}{2}}2^{k-1}
\Gamma(\nu+1)}\sum_{n=0}^{\infty}\frac{r_{k}(n)}{\sqrt{n+\a}\sqrt{n+\b}}
\left(\frac{\sqrt{n+\a}-\sqrt{n+\b}}{\sqrt{n+\a}+\sqrt{n+\b}}\right)^{\nu}\nonumber\\
&\quad\times\left(\frac{1}{\sqrt{n+\a}}+\frac{1}{\sqrt{n+\b}}\right)^{k-2}\pFq21{\nu+1-\frac{k}{2},
  1-\frac{k}{2}}{\nu+1}{\left(\frac{\sqrt{n+\a}-\sqrt{n+\b}}{\sqrt{n+\a}+\sqrt{n+\b}}\right)^{2}},
\end{align*}
with the $n=0$ term of the series on the left-hand side interpreted as the limit in \eqref{n0lim}.\\
\end{remark}

\begin{remark} We can generalize Theorem \ref{popgenrk} by replacing the
coefficients
$r_k(n)$ by any arithmetical function $a(n)$ generated by a Dirichlet series
satisfying a functional equation of the form \eqref{funcequa}.
%We obtain the following identity of Popov \cite{popov1935} as a special case
%of Theorem \ref{popgenrk}.
\end{remark}

\begin{corollary}
For \textup{Re}$(\nu)>0$, Popov's identity \eqref{popovid} holds for a much larger region \textup{Re}$(\sqrt{\a})\geq$ \textup{Re}$(\sqrt{\b})>0$.
%\begin{align}\label{popovid1}
%&\frac{2\pi}{(\a-\b)^{\nu}}\sum_{n=1}^{\infty}r_2(n)I_{\nu}(\pi(\sqrt{n\a}-\sqrt{n\b}))K_{\nu}(\pi(\sqrt{n\a}+\sqrt{n\b}))\nonumber\\
%&=\frac{\nu-\pi\sqrt{\a\b}}{\nu\sqrt{\a\b}\left(\sqrt{\a}+\sqrt{\b}\right)^{2\nu}}
%+\sum_{n=1}^{\infty}\frac{r_2(n)}{\sqrt{n+\a}\sqrt{n+\b}\left(\sqrt{n+\a}+\sqrt{n+\b}\right)^{2\nu}}.
%\end{align}
\end{corollary}

\begin{proof}
Set $k=2$ in Theorem \ref{popgenrk} and simplify using \eqref{n0lim}.
\end{proof}

%\textbf{Remark.} Popov gives the condition $\a\geq\b>0$ for the validity of the above result, however, we now see that it is true for the much larger region Re$(\sqrt{\a})\geq$ Re$(\sqrt{\b})>0$.
%Another corollary of the above theorem we obtain is when $\nu=1/2$.
\begin{corollary}\label{cor3.2}
For any positive integer $k\geq 2$ and \textup{Re}$(\sqrt{\a})\geq$ \textup{Re}$(\sqrt{\b})>0$,
\begin{align}\label{popgenrkhalf}
&\sum_{n=1}^{\infty}\frac{r_{k}(n)}{\sqrt{n}}e^{-\pi\sqrt{n}(\sqrt{\a}+\sqrt{\b})}\sinh(\pi\sqrt{n}(\sqrt{\a}-\sqrt{\b}))
\\
&=-\pi\left(\sqrt{\a}-\sqrt{\b}\right)-\frac{1}{2\pi^{(k-1)/2}}
\Gamma\left(\dfrac{k-1}{2}\right)\sum_{n=0}^{\infty}r_k(n)\left((n+\a)^{(1-k)/2}-(n+\b)^{(1-k)/2}\right).\notag
\end{align}
\end{corollary}
%When $k=2$, this gives, as a special case, a formula of Popov \cite[Equation (11)]{popov1935}.
%\begin{corollary}

%\end{corollary}
\begin{proof}
Let $\nu=\tf12$ in Theorem \ref{popgenrk}. From \cite[p.~80, eq.~(10)]{watson-1966a},
%From \cite[p.~924, eq.~\textbf{8.464.1}]{grn} and  \cite[p.~911, eq.~\textbf{8.406.1}]{grn},
\begin{equation}\label{ispl}
I_{\frac{1}{2}}(\pi(\sqrt{n\a}-\sqrt{n\b}))=\frac{\sqrt{2}}{\pi}\frac{\sinh(\pi(\sqrt{n\a}-\sqrt{n\b}))}{\sqrt{\sqrt{n\a}-\sqrt{n\b}}},
\end{equation}
and from \eqref{kspl},
\begin{equation}\label{kspla}
K_{\frac{1}{2}}(\pi(\sqrt{n\a}+\sqrt{n\b}))=\frac{1}{\sqrt{2(\sqrt{n\a}+\sqrt{n\b})}}e^{-\pi(\sqrt{n\a}+\sqrt{n\b})}.
\end{equation}
Next, from \cite[p.~389, eq.~(107)]{prudv3},
\begin{equation}\label{dizzy}
\pFq21{a,a+\frac{1}{2}}{\frac{3}{2}}{z}=\frac{1}{2(2a-1)\sqrt{z}}\left\{(1-\sqrt{z})^{1-2a}-(1+\sqrt{z})^{1-2a}\right\}.
\end{equation}
Let $a=1-\frac{k}{2}$ and
$z=\left(\frac{\sqrt{x+\a}-\sqrt{x+\b}}{\sqrt{x+\a}+\sqrt{x+\b}}\right)^2$ in
\eqref{dizzy}  and simplify to find that
\begin{align}\label{hspl}
&\pFq21{1-\frac{k}{2},\frac{3-k}{2}}{\frac{3}{2}}{\left(\frac{\sqrt{x+\a}-\sqrt{x+\b}}{\sqrt{x+\a}+\sqrt{x+\b}}\right)^2}\nonumber\\
&=\frac{2^{k-2}}{(1-k)}\frac{\left(\sqrt{x+\a}+\sqrt{x+\b}\right)^{2-k}\left((x+\b)^{(k-1)/2}
-(x+\a)^{(k-1)/2}\right)}{\left(\sqrt{x+\a}-\sqrt{x+\b}\right)}.
\end{align}
Now use \eqref{ispl}, \eqref{kspla} and \eqref{hspl} in \eqref{popgenrkeqn} to arrive at \eqref{popgenrkhalf}.
\end{proof}

\begin{remark} The special case $k=2$ of Corollary \ref{cor3.2} is equivalent to the following identity of Popov \cite[eq.~(11)]{popov1935}:
\begin{align*}
&\frac{1}{\sqrt{\b}}-\frac{1}{\sqrt{\a}}+\sum_{n=1}^{\infty}r_2(n)\left(\frac{1}{\sqrt{n+\b}}-\frac{1}{\sqrt{n+\a}}\right)\nonumber\\
&=2\pi(\sqrt{\a}-\sqrt{\b})+\sum_{n=1}^{\infty}\frac{r_2(n)}{\sqrt{n}}\left(e^{-2\pi\sqrt{n\b}}-e^{-2\pi\sqrt{n\a}}\right).
\end{align*}
This identity should be compared with \eqref{ramdixfer}.
\end{remark}
% However, it does not appear that  one can easily be obtained from the other.
%A yet another corollary of the identity in the above theorem is obtained
%when we divide both sides by $(\a-\b)^{\nu}$ and then let $\a\to\b^{+}$:

\begin{corollary}\label{cordfpop}
For any positive integer $k>1$, \textup{Re}$(\sqrt{\b})>0$, and \textup{Re}$(\nu)>0$,
\begin{equation}\label{r2dix1}
\sum_{n=0}^{\infty}r_k(n)n^{\frac{\nu}{2}}K_{\nu}(2\pi\sqrt{n\b})
=\frac{\b^{\frac{\nu}{2}}\G\left(\nu+\frac{k}{2}\right)}{2\pi^{\nu+\frac{k}{2}}}
\sum_{n=0}^{\infty}\frac{r_{k}(n)}{(n+\b)^{\nu+\frac{k}{2}}}.
\end{equation}
\end{corollary}

\begin{proof}
Divide both sides of \eqref{popgenrkeqn} by $(\a-\b)^{\nu}$ and then let
$\a\to\b^{+}$. By  \eqref{besseli},
\begin{equation*}
\lim_{\a\to\b^{+}}\frac{I_{\nu}(\pi(\sqrt{n\a}-\sqrt{n\b}))K_{\nu}(\pi(\sqrt{n\a}+\sqrt{n\b}))}{(\a-\b)^{\nu}}
=\left(\frac{\pi}{4}\right)^{\nu}\left(\frac{n}{\b}\right)^{\frac{\nu}{2}}\frac{K_{\nu}(2\pi\sqrt{n\b})}{\Gamma(1+\nu)}.
\end{equation*}
Thus,
\begin{align}\label{b}
\frac{(\pi/4)^{\nu}}{\b^{\frac{\nu}{2}}\G(1+\nu)}\sum_{n=1}^{\infty}r_k(n)n^{\frac{\nu}{2}}K_{\nu}(2\pi\sqrt{n\b})
=-\frac{1}{\nu 2^{2\nu+1}\b^{\nu}}+\frac{\G(\nu+\frac{k}{2})}{\pi^{\frac{k}{2}}2^{2\nu+1}\G(\nu+1)}\sum_{n=0}^{\infty}
\frac{r_{k}(n)}{(n+\b)^{\nu+\frac{k}{2}}}.
\end{align}
Multiply both sides of \eqref{b}  by
$\b^{\frac{\nu}{2}}\G(1+\nu)/(\pi/4)^{\nu}$. By \eqref{besselk} and \eqref{besseli},
\begin{equation*}
\lim_{x\to 0}x^{\nu/2}K_{\nu}(2\pi\sqrt{x\b})=\frac{\G(\nu)}{2\pi^{\nu}\b^{\nu/2}}.
\end{equation*}
 Thus, the first term on the  right-hand side of \eqref{b} can be interpreted
 as the additive inverse of the $n=0$ term of the series on the  left-hand
 side of \eqref{b}. Hence, we arrive at \eqref{r2dix1}.
\end{proof}

Corollary \ref{cordfpop} was also established by Popov \cite[eq.~(6)]{popov1935}. As a special case, Corollary \ref{cordfpop}  implies the following formula in Corollary \ref{c}, which specializes, when $k=2$, to a identity of Hardy \cite[eq.~(2.12)]{hardyqjpam1915}, which he used in his study of the famous circle problem to prove that
$$\sum_{n\leq x}r_2(n)-\pi x=\Omega(x^{1/4}).$$
Corollary \ref{c} was also obtained by Popov \cite[eq.~(1)]{ocs}, but (with our $k$ replaced by his $\nu$) his condition $\text{Re}(\nu)>-\tfrac12$ should be replaced by $ \nu>0$.

\begin{corollary}\label{c}
For any positive integer $k>1$ and \textup{Re}$(\sqrt{\b})>0$,
\begin{equation}\label{hardygen}
\sum_{n=0}^{\infty}r_k(n)e^{-2\pi\sqrt{n\b}}=\sqrt{\b}\frac{\G\left(\frac12(k+1)\right)}{\pi^{(k+1)/2}}
\sum_{n=0}^{\infty}\frac{r_k(n)}{(n+\b)^{(k+1)/2}}.
\end{equation}
\end{corollary}

\begin{proof}
Let $\nu=\tf12$ in Corollary \ref{cordfpop} and use \eqref{kspl} to obtain \eqref{hardygen} after simplification.
\end{proof}

\begin{corollary}
Let, for $0\leq|k|<1$,
\begin{align*}
K(k):=\int_{0}^{\frac{\pi}{2}}\frac{dt}{\sqrt{1-k^2\sin^{2}t}}
\end{align*}
be the complete elliptic integral of the first kind and
\begin{align*}
D(k):=\int_{0}^{\frac{\pi}{2}}\frac{\sin^{2}t\, dt}{\sqrt{1-k^2\sin^{2}t}}.
\end{align*}
Then, for \textup{Re}$(\sqrt{\a})\geq$ \textup{Re}$(\sqrt{\b})>0$,
\begin{align}\label{elliptickd}
&\sum_{n=1}^{\infty}r_3(n)I_{1}(\pi(\sqrt{n\a}-\sqrt{n\b}))K_{1}(\pi(\sqrt{n\a}+\sqrt{n\b}))\nonumber\\
&=-\frac{1}{2}\left(\frac{\sqrt{\a}-\sqrt{\b}}{\sqrt{\a}+\sqrt{\b}}\right)+\frac{(\a-\b)}{4\pi^2}
\sum_{n=0}^{\infty}\frac{r_3(n)}{(n+\a)(n+\b)}\nonumber\\
&\quad\times\left\{\frac{1}{\sqrt{n+\a}}K\left(\frac{\sqrt{\a-\b}}{\sqrt{n+\a}}\right)
-\frac{4n+2\a+2\b}{\left(\sqrt{n+\a}+\sqrt{n+\b}\right)^3}D\left(\frac{\sqrt{n+\a}-\sqrt{n+\b}}{\sqrt{n+\a}+\sqrt{n+\b}}\right)\right\}.
\end{align}
\end{corollary}

\begin{proof}
Let $k=3$ and $\nu=1$ in Theorem \ref{popgenrk}. Using the identity \cite[p.~395]{prudv3}
\begin{equation*}
\pFq21{-\frac{1}{2}, \frac{1}{2}}{2}{z}=\frac{4}{3\pi}\left(2K(\sqrt{z})-(1+z)D(\sqrt{z})\right),
\end{equation*}
and Landen's transformation \cite[p.~112]{ntsr}
\begin{align*}
K(k)=\frac{1}{(1+k)}K\left(\frac{2\sqrt{k}}{1+k}\right),\hspace{5mm}0\leq k<1,
\end{align*}
we arrive at \eqref{elliptickd}.
\end{proof}

\section{Analytically Continuing Theorem \ref{popgenrk} to R\MakeLowercase{e}$(\nu)>-1$}\label{ac}
Recall the definition of $\zeta_k(s)$ from \eqref{zeta} and the fact that it has an analytic continuation into the entire complex plane. Also, recall \eqref{popgenrkeqn}.
Note that for Re$(\sqrt{\a})\geq$ Re$(\sqrt{\b})>0$ and Re$(\nu)>0$, if we separate the $n=0$ term on the right-hand side of \eqref{popgenrkeqn} and use the aforementioned analytic continuation, we can rewrite \eqref{popgenrkeqn} in the form
{\allowdisplaybreaks\begin{align}\label{popgenrkeqnac}
&\sum_{n=1}^{\infty}r_k(n)I_{\nu}(\pi(\sqrt{n\a}-\sqrt{n\b}))K_{\nu}(\pi(\sqrt{n\a}+\sqrt{n\b}))\nonumber\\
&=-\frac{1}{2\nu}\left(\frac{\sqrt{\a}-\sqrt{\b}}{\sqrt{\a}+\sqrt{\b}}\right)^{\nu}
+\frac{2^{-2\nu-1}(\a-\b)^{\nu}\G\left(\nu+\frac{1}{2}k\right)}{\pi^{k/2}\G(\nu+1)}
\zeta_k(\nu+\tfrac{k}{2})\nonumber\\
&\quad+\frac{\Gamma\left(\nu+\frac{k}{2}\right)}{\pi^{\frac{k}{2}}2^{k-1}\Gamma(\nu+1)\sqrt{\a\b}}
\left(\frac{\sqrt{\a}-\sqrt{\b}}{\sqrt{\a}+\sqrt{\b}}\right)^{\nu}\left(\frac{1}{\sqrt{\a}}
+\frac{1}{\sqrt{\b}}\right)^{k-2}\notag\\&\quad\times\pFq21{\nu+1-\frac{k}{2}, 1-\frac{k}{2}}{\nu+1}{\left(\frac{\sqrt{\a}-\sqrt{\b}}{\sqrt{\a}+\sqrt{\b}}\right)^{2}}\nonumber\\
&\quad+\frac{\Gamma\left(\nu+\frac{k}{2}\right)}{\pi^{\frac{k}{2}}2^{k-1}\Gamma(\nu+1)}
\sum_{n=1}^{\infty}r_{k}(n)\bigg\{\frac{1}{\sqrt{n+\a}\sqrt{n+\b}}
\left(\frac{\sqrt{n+\a}-\sqrt{n+\b}}{\sqrt{n+\a}+\sqrt{n+\b}}\right)^{\nu}\nonumber\\
&\quad\times\left(\frac{1}{\sqrt{n+\a}}+\frac{1}{\sqrt{n+\b}}\right)^{k-2}\pFq21{\nu+1-\frac{k}{2}, 1-\frac{k}{2}}{\nu+1}{\left(\frac{\sqrt{n+\a}-\sqrt{n+\b}}{\sqrt{n+\a}+\sqrt{n+\b}}\right)^{2}}\nonumber\\
&\quad\quad-\frac{2^{k-2-2\nu}(\a-\b)^{\nu}}{n^{\nu+k/2}}\bigg\}.
\end{align}}
Now for large $n$, from the definition \eqref{2f1} of $_2F_1$,
\begin{align}\label{ases}
&\frac{1}{\sqrt{n+\a}\sqrt{n+\b}}\left(\frac{\sqrt{n+\a}-\sqrt{n+\b}}{\sqrt{n+\a}+\sqrt{n+\b}}\right)^{\nu}
\left(\frac{1}{\sqrt{n+\a}}+\frac{1}{\sqrt{n+\b}}\right)^{k-2}\nonumber\\
&\times\pFq21{\nu+1-\frac{k}{2}, 1-\frac{k}{2}}{\nu+1}{\left(\frac{\sqrt{n+\a}-\sqrt{n+\b}}{\sqrt{n+\a}+\sqrt{n+\b}}\right)^{2}}\notag\\
&=\frac{2^{k-2-2\nu}(\a-\b)^{\nu}}{n^{\nu+k/2}}+O_{\a, \b, \nu, k}\left(\frac{1}{n^{\nu+\frac{1}{2}k+1}}\right).
\end{align}
The main term in this asymptotic estimate was already obtained in \eqref{fas}.
From \eqref{ases}, we see that the series on the right-hand side of \eqref{popgenrkeqnac} converges for Re$(\nu)>-1$.  By analytic continuation, we conclude that \eqref{popgenrkeqnac} is valid for $\text{Re }\nu>-1.$  Thus, letting $\nu=0$ in \eqref{popgenrkeqnac}, we find that
{\allowdisplaybreaks\begin{align}\label{popgenrkeqnac0}
&\sum_{n=1}^{\infty}r_k(n)I_{0}(\pi(\sqrt{n\a}-\sqrt{n\b}))K_{0}(\pi(\sqrt{n\a}+\sqrt{n\b}))\nonumber\\
&=\lim_{\nu\to 0}\left(\frac{2^{-2\nu-1}(\a-\b)^{\nu}\G\left(\nu+\frac{1}{2}k\right)}{\pi^{k/2}\G(\nu+1)}
\zeta_k(\nu+\tfrac{k}{2})-\frac{1}{2\nu}\left(\frac{\sqrt{\a}-\sqrt{\b}}{\sqrt{\a}+\sqrt{\b}}\right)^{\nu}\right)
\nonumber\\
&\quad+\frac{\Gamma\left(\frac{k}{2}\right)}{\pi^{\frac{k}{2}}2^{k-1}\sqrt{\a\b}}
\left(\frac{1}{\sqrt{\a}}+\frac{1}{\sqrt{\b}}\right)^{k-2}\pFq21{1-\frac{k}{2}, 1-\frac{k}{2}}{1}{\left(\frac{\sqrt{\a}-\sqrt{\b}}{\sqrt{\a}+\sqrt{\b}}\right)^{2}}\nonumber\\
&\quad+\frac{\Gamma\left(\frac{k}{2}\right)}{\pi^{\frac{k}{2}}2^{k-1}}\sum_{n=1}^{\infty}r_{k}(n)
\Bigg\{\frac{1}{\sqrt{n+\a}\sqrt{n+\b}}\left(\frac{1}{\sqrt{n+\a}}+\frac{1}{\sqrt{n+\b}}\right)^{k-2}\nonumber\\
&\quad\quad\quad\quad\quad\times\pFq21{1-\frac{k}{2}, 1-\frac{k}{2}}{1}{\left(\frac{\sqrt{n+\a}-\sqrt{n+\b}}{\sqrt{n+\a}+\sqrt{n+\b}}\right)^{2}}-\frac{2^{k-2}}{n^{k/2}}\Bigg\}.
\end{align}}%

 %For any positive integer $k\geq 2$, as long as the Dirichlet series $\sum_{n=1}^{\infty}r_k(n)n^{-\nu-k/2}$ occurring in \eqref{popgenrkeqnac} has a representation of the type in \eqref{dsrk2}--\eqref{dsrk8}, that is one which is analytic for Re$(\nu)>-1$, it is clear that, by the principle of analytic continuation, \eqref{popgenrkeqnac} is valid for Re$(\nu)>-1$. In particular, for $k=2, 4, 6,$ and $8$, using the formulas in \eqref{dsrk2}--\eqref{dsrk8} to represent $\sum_{n=1}^{\infty}r_k(n)n^{-\nu-k/2}$,
We now specialize \eqref{popgenrkeqnac0}.  The Dirichlet series for $r_k(n), k=2, 4, 6$ and $8$, are given by \cite[p.~102]{borweinchoi1}
{\allowdisplaybreaks\begin{align}
	\sum_{n=1}^{\infty}\frac{r_2(n)}{n^s}&=4\zeta(s)\mathfrak{B}(s), \hspace{3mm}\text{Re}(s)>1,\label{dsrk2}\\
	\sum_{n=1}^{\infty}\frac{r_4(n)}{n^s}&=8(1-4^{1-s})\zeta(s)\zeta(s-1), \hspace{3mm}\text{Re}(s)>2,\label{dsrk4}\\
	\sum_{n=1}^{\infty}\frac{r_6(n)}{n^s}&=16\zeta(s-2)\mathfrak{B}(s)-4\zeta(s)\mathfrak{B}(s-2), \hspace{3mm}\text{Re}(s)>3,\\
	\sum_{n=1}^{\infty}\frac{r_8(n)}{n^s}&=16(1-2^{1-s}+4^{2-s})\zeta(s)\zeta(s-3), \hspace{3mm}\text{Re}(s)>4,\label{dsrk8}
	\end{align}}%
where $\zeta(s)$ denotes the Riemann zeta function and $\mathfrak{B}(s)$ denotes the Dirichlet beta-function, defined for Re$(s)>0$, by
\begin{equation*} 	
\mathfrak{B}(s):=\sum_{n=0}^{\infty}\frac{(-1)^n}{(2n+1)^{s}}.
\end{equation*}
The function $\mathfrak{B}(s)$ is a Dirichlet $L$-function, and consequently it satisfies the functional equation \cite[p.~69]{davenport}
\begin{equation}\label{dbffunc}
\mathfrak{B}(1-s)=\left(\frac{2}{\pi}\right)^s\sin\left(\frac{\pi s}{2}\right)\G(s)\mathfrak{B}(s).
\end{equation}
For these four values of $k$, we can deduce the following four theorems. Since the line of reasoning is similar in the proofs of Theorems \ref{popgenrkac2} and \ref{popgenrkac6}, we provide the proof of only the former. Similarly, we prove only Theorem \ref{popgenrkac4}, because its proof is similar to that of Theorem \ref{popgenrkac8}.

\begin{theorem}\label{popgenrkac2}
Let $\gamma$ denote Euler's constant. For \textup{Re}$(\sqrt{\a})\geq$ \textup{Re}$(\sqrt{\b})>0$,
\begin{align}\label{popgenrkeqnac2}
&\sum_{n=1}^{\infty}r_2(n)I_{0}(\pi(\sqrt{n\a}-\sqrt{n\b}))K_{0}(\pi(\sqrt{n\a}+\sqrt{n\b}))\nonumber\\
&=\frac{1}{2\pi\sqrt{\a\b}}+\g+\log\left(\frac{\sqrt{\a}+\sqrt{\b}}{2}\right)+\frac{1}{2}\log\left(\frac{\pi}{2}\right)
-\mathfrak{B}'(0)\nonumber\\
&\quad+\frac{1}{2\pi}\sum_{n=1}^{\infty}r_2(n)\left(\frac{1}{(n+\a)^{1/2}(n+\b)^{1/2}}-\frac{1}{n}\right).
\end{align}
\end{theorem}

\begin{proof}
Letting $k=2$ in \eqref{popgenrkeqnac0} and using \eqref{dsrk2}, we find that
\begin{align*}
&\sum_{n=1}^{\infty}r_2(n)I_{0}(\pi(\sqrt{n\a}-\sqrt{n\b}))K_{0}(\pi(\sqrt{n\a}+\sqrt{n\b}))\nonumber\\
&=L+\frac{1}{2\pi\sqrt{\a\b}}+\frac{1}{2\pi}\sum_{n=1}^{\infty}r_2(n)\left(\frac{1}{(n+\a)^{1/2}(n+\b)^{1/2}}-\frac{1}{n}\right),
\end{align*} 	
where
\begin{align*}
L=\lim_{\nu\to 0}\left\{\frac{2^{1-2\nu}}{\pi}(\a-\b)^{\nu}\zeta(\nu+1)\mathfrak{B}(\nu+1)-\frac{1}{2\nu}
\left(\frac{\sqrt{\a}-\sqrt{\b}}{\sqrt{\a}+\sqrt{\b}}\right)^{\nu}\right\}.
\end{align*}
Using the Madhava-Gregory series for $\pi/4=\mathfrak{B}(1)$ and the well-known limit \cite[p.~16]{titch}
\begin{equation}\label{riepol}
\lim_{s\to 1}\left(\zeta(s)-\frac{1}{s-1}\right)=\gamma,
\end{equation}
we see that
\begin{align*}
L=\frac{\g}{2}+\frac{1}{2}\lim_{\nu\to 0}\frac{\left(\pi^{-1}2^{2-2\nu}(\sqrt{\a}+\sqrt{\b})^{2\nu}\mathfrak{B}(\nu+1)-1\right)}{\nu}.
\end{align*}
Employing L'Hopital's rule, \eqref{dbffunc}, $\mathfrak{B}(1)=\pi/4$ once again, and $\Gamma^{\prime}(1)=-\gamma$, upon simplifying, we arrive at
\begin{align*}
L=\frac{\g}{2}+\log\left(\frac{\sqrt{\a}+\sqrt{\b}}{2}\right)+\frac{2}{\pi}\mathfrak{B}'(1)
=\g+\log\left(\frac{\sqrt{\a}+\sqrt{\b}}{2}\right)+\frac{1}{2}\log\left(\frac{\pi}{2}\right)-\mathfrak{B}'(0).	
\end{align*}
 This proves \eqref{popgenrkeqnac2}.
\end{proof}

\begin{remark} When we let $\a\to\b^{+}$ in Theorem \ref{popgenrkac2},  use the facts that $I_{0}(0)=1$ and \cite[p.~20]{titch} $\zeta^{\prime}(0)=-\tf12\log(2\pi)$, and replace $\b$ by $x$, we obtain a result of Dixon and Ferrar \cite[eq.~(3.24)]{dixfer2}
\begin{equation}\label{DF}
2\sum_{n=1}^{\infty}r_2(n)K_0(2\pi\sqrt{nx})-\log(\tf12\pi x)-2\gamma+2\mathfrak{B}^{\prime}(0)=\df{1}{\pi x}+\df{1}{\pi}\sum_{n=1}^{\infty}r_2(n)\left\{\df{1}{x+n}-\df{1}{n}\right\}.
\end{equation}
\end{remark}

\begin{theorem}\label{popgenrkac4}
For \textup{Re}$(\sqrt{\a})\geq$ \textup{Re}$(\sqrt{\b})>0$,
\begin{align}\label{popgenrkeqnac4}
&\sum_{n=1}^{\infty}r_4(n)I_{0}(\pi(\sqrt{n\a}-\sqrt{n\b}))K_{0}(\pi(\sqrt{n\a}+\sqrt{n\b}))\nonumber\\
&=\frac{\a+\b}{4\pi^2(\a\b)^{3/2}}+\frac{\g}{2}+\frac{1}{2}\left(1-\frac{2}{3}\log 4+2\log\left(\sqrt{\a}+\sqrt{\b}\right)+\frac{6}{\pi^2}\zeta'(2)\right)\nonumber\\
&\quad+\frac{1}{4\pi^2}\sum_{n=1}^{\infty}r_4(n)\left\{\frac{2n+\a+\b}{(n+\a)^{3/2}(n+\b)^{3/2}}-\frac{2}{n^2}\right\}.
\end{align}
\end{theorem}

\begin{proof}
Let $k=4$ in \eqref{popgenrkeqnac0} and use \eqref{dsrk4}. We note that
 $$\pFq21{-1, -1}{1}{x}=1+x.$$
 We will calculate the limit below in two stages.  In the first, we use \eqref{riepol} and the evaluation $\zeta(2)=\pi^2/6$. In the second, we  employ L'Hopital's rule and again use the evaluation $\zeta(2)=\pi^2/6$.  Accordingly,
\begin{align*}
&\lim_{\nu\to 0}\left\{\frac{(\a-\b)^{\nu}\G(\nu+2)}{2^{2\nu+1}\pi^2\G(\nu+1)}8(1-4^{-\nu-1})\zeta(\nu+2)\zeta(\nu+1)-\frac{1}{2\nu}
\left(\frac{\sqrt{\a}-\sqrt{\b}}{\sqrt{\a}+\sqrt{\b}}\right)^{\nu}\right\}\nonumber\\
&=\frac{\g}{2}+\frac{1}{2}\lim_{\nu\to 0}\frac{\left(8\pi^{-2}2^{-2\nu}(\nu+1)(1-4^{-\nu-1})(\sqrt{\a}+\sqrt{\b})^{2\nu}\zeta(\nu+2)-1\right)}{\nu}\nonumber\\
&=\frac{\g}{2}+\frac{1}{2}\left(1-\frac{2}{3}\log 4+2\log\left(\sqrt{\a}+\sqrt{\b}\right)+\frac{6}{\pi^2}\zeta'(2)\right).
\end{align*}
This gives \eqref{popgenrkeqnac4}.
\end{proof}

\begin{theorem}\label{popgenrkac6}
For \textup{Re}$(\sqrt{\a})\geq$ \textup{Re}$(\sqrt{\b})>0$,
\begin{align*}
&\sum_{n=1}^{\infty}r_6(n)I_{0}(\pi(\sqrt{n\a}-\sqrt{n\b}))K_{0}(\pi(\sqrt{n\a}+\sqrt{n\b}))\nonumber\\
&=\frac{3\a^2+2\a\b+3\b^2}{8\pi^3(\a\b)^{5/2}}+\frac{\g}{2}-\frac{\zeta(3)}{\pi^2}
+\log\left(\frac{\sqrt{\a}+\sqrt{\b}}{2}\right)+\frac{3}{4}+\frac{16}{\pi^3}\mathfrak{B}'(3)\nonumber\\
&\quad+\frac{1}{8\pi^3}\sum_{n=1}^{\infty}r_6(n)\left\{\frac{8n^2+8n\a+3\a^2+8n\b+2\a\b+3\b^2}{(n+\a)^{5/2}(n+\b)^{5/2}}
-\frac{8}{n^3}\right\}.
\end{align*}
\end{theorem}

\begin{theorem}\label{popgenrkac8}
For \textup{Re}$(\sqrt{\a})\geq$ \textup{Re}$(\sqrt{\b})>0$,
\begin{align*}
&\sum_{n=1}^{\infty}r_8(n)I_{0}(\pi(\sqrt{n\a}-\sqrt{n\b}))K_{0}(\pi(\sqrt{n\a}+\sqrt{n\b}))\nonumber\\
&=\frac{3(5\a^3+3\a^2\b+3\a\b^2+5\b^3)}{16\pi^4(\a\b)^{7/2}}+\frac{\g}{2}+\log\left(\frac{\sqrt{\a}
+\sqrt{\b}}{2}\right)+\frac{11}{12}+\frac{45}{\pi^4}\zeta'(4)\nonumber\\
&\quad+\frac{3}{16\pi^4}\sum_{n=1}^{\infty}r_8(n)\left\{\frac{P(n, \a, \b)}{(n+\a)^{7/2}(n+\b)^{7/2}}-\frac{16}{n^4}\right\},
\end{align*}
where $P(n, \a, \b)$ is a polynomial in $n$ (as well as in $\a$ and $\b$) given by
\begin{align*}
P(n, \a, \b):=16n^3+24n^2\a+18n\a^2+5\a^3+24n^2\b+12n\a\b+3\a^2\b+18n\b^2+3\a\b^2+5\b^3.
\end{align*}
\end{theorem}

As previously mentioned, letting $\a\to\b^{+}$ in Theorems \ref{popgenrkac4}--\ref{popgenrkac8}, we obtain analogues of Dixon and Ferrar's identity \eqref{DF}.
%result alluded to in the remark after Theorem \ref{popgenrkac2}.
 We refrain from stating them explicitly since it is simple to derive them from the theorems above.

We now discuss one further interesting special case of \eqref{popgenrkeqnac}. Let $k=2m, 1\leq m\leq 4$. Substitute \eqref{dsrk2}--\eqref{dsrk8} in \eqref{popgenrkeqnac}, according as $m=1$, $2$, $3$, or $4$. Then let $\nu=-\tf12$; the required  formulas for Bessel functions with arguments $\pm\tf12$ can be found in Watson's \emph{Treatise} \cite[pp.~53, 79, 80]{watson-1966a}. We also need the formula \cite[p.~389, eq.~(106)]{prudv3}  $$\pFq21{a,a+\frac{1}{2}}{\frac{1}{2}}{x}=\frac{1}{2}\left\{(1+\sqrt{x})^{-2a}+(1-\sqrt{x})^{-2a}\right\}.$$
%the fact \cite[p.~924, Formula \textbf{8.464.2}]{grn} and \cite[p.~911, Formula \textbf{8.406.1}]{grn} %$I_{-\frac{1}{2}}(z)=\sqrt{\frac{2}{\pi z}}\cosh(z)$ and the fact \cite[p.~925, Formula \textbf{8.469.3}]{grn} to
 Foregoing all further details, we arrive at the following theorem.

\begin{theorem}\label{goody}
For \textup{Re}$(\sqrt{\a})\geq$ \textup{Re}$(\sqrt{\b})>0$, and a positive integer $m$ such that $1\leq m\leq 4$,
\begin{align*}
&\sum_{n=1}^{\infty}\frac{r_{2m}(n)}{\sqrt{n}}e^{-\pi(\sqrt{n\a}+\sqrt{n\b})}\cosh(\pi(\sqrt{n\a}-\sqrt{n\b}))\nonumber\\
&=a_m+\pi(\sqrt{\a}+\sqrt{\b})+\frac{\G\left(m-\frac{1}{2}\right)}{2\pi^{m-\frac{1}{2}}}\left(\frac{1}{\a^{(2m-1)/2}}
+\frac{1}{\b^{(2m-1)/2}}\right)\nonumber\\
&\quad+\frac{\G\left(m-\frac{1}{2}\right)}{2\pi^{m-\frac{1}{2}}}\sum_{n=1}^{\infty}r_{2m}(n)
\left\{\frac{1}{(n+\a)^{(2m-1)/2}}+\frac{1}{(n+\b)^{(2m-1)/2}}-\frac{2}{n^{(2m-1)/2}}\right\},
\end{align*}
where
\begin{align*}
a_1&:=4\zeta\left(\frac{1}{2}\right)\mathfrak{B}\left(\frac{1}{2}\right),\nonumber\\
a_2&:=\frac{2}{\pi}\zeta\left(\frac{3}{2}\right)\zeta\left(\frac{1}{2}\right),\nonumber\\
a_3&:=\frac{3}{4\pi^2}\left(16\zeta\left(\frac{1}{2}\right)\mathfrak{B}\left(\frac{5}{2}\right)-4\zeta\left(\frac{5}{2}\right)
\mathfrak{B}\left(\frac{1}{2}\right)\right),\nonumber\\
a_4&:=\frac{30}{\pi^3}\left(\frac{9}{8}-\frac{1}{2^{5/2}}\right)\zeta\left(\frac{7}{2}\right)\zeta\left(\frac{1}{2}\right).
\end{align*}
\end{theorem}
\begin{remark} We can let $\a\to\b^{+}$ in Theorem \ref{goody} to obtain interesting special cases.\end{remark}

%\textbf{Remark 2.} We can obtain further results of the type given above for higher values of $k$, especially when $k=12$ and $24$, %that is, when $m=6$ and $12$, respectively. However, the Dirichlet series representations for $k=24$ and $k=12$ %\cite[p.~126]{borweinchoi1}, respectively, involve the Ramanujan tau-function $\tau(n)$ and the function $\omega(n)$ defined by
%\begin{equation*}
%\sum_{n=1}^{\infty}\omega(n)q^n:=q\prod_{m=1}^{\infty}(1-q^{2m})^{12},
%\end{equation*}
%and the calculations become too unwieldy, hence we do not pursue them here.\\

\section{Concluding Remarks and Further Possible Work}\label{conclusion}
%A remarkable feature of Theorem \ref{popgenrk} is that it is elegant despite having four different parameters, and has many well-known %results in the literature as its special cases. The result that enables us in obtaining this nice transformation is Koshliakov's %beautiful integral evaluation \eqref{koshfock}, which is a generalization of Fock's integral evaluation \eqref{fock}. This demonstrates %how useful the integral evaluations involving Bessel functions can be in analytic number theory. In \cite{dixfer1933}, Dixon and Ferrar %obtained integral representations for the product $I_{\nu}(x)K_{\mu}(X)$, where the orders $\mu$ and $\nu$ of the Bessel functions are %not necessarily equal. It may be worthwhile seeing if any of these representations possibly leads to a more general transformation.
At first sight,  Theorem \ref{popgenrk} does not appear remarkable.  However, as we have seen, several elegant and well-known transformations in the literature are special cases of Theorem \ref{popgenrk}.
%Also in light of Hardy's application of Corollary \ref{c} to obtain the famous omega result in his study on the circle problem, our general theorem, namely Theorem \ref{popgenrk}, may also be applicable in further studies on the circle problem.

 The summands in \eqref{popgenrkeqn} contain a product of Bessel functions $I_{\nu}(X)$ and $K_{\nu}(x)$. Dixon and Ferrar \cite{dixfer1933} obtained an integral representations for the product $I_{\mu}(X)K_{\nu}(x)$, where the orders $\mu$ and $\nu$  are not necessarily equal.  Therefore, perhaps there exists a more general transformation than that in Theorem \ref{popgenrk}.

\begin{center}
\textbf{Acknowledgements}
\end{center}

\noindent
The second author's research was partially supported by the DST grant ECR/2015/000070. The authors sincerely thank Dmitry Vasilenko, Vice-Rector for International Relations at St.~Petersburg State University of Economics, for sending them a list of the  $12$ publications of A. I. Popov. They also thank Arindam Roy for sending them a copy of \cite{popov1935}, Anton Lukyanenko for translating for them a section of that paper, and Jeremy Rouse for discussions on bounds for $r_k(n)$.

%\subsection{Sums of odd number of squares}
%If we substitute $k=3, \nu=-1/2$ in \eqref{popgenrkeqnac}, and then make use of \cite[]{}

\end{document}